\documentclass[10pt,leqno]{amsart}
\usepackage{amssymb}
\newcommand{\Q}{{\mathbb{Q}}}

\newcommand{\C}{{\mathbb{C}}}
\newcommand{\Z}{{\mathbb{Z}}}
\newcommand{\R}{{\mathbb{R}}}

\newcommand{\ba}{{\boldsymbol{a}}}
\newcommand{\bc}{{\boldsymbol{c}}}

\newcommand{\bD}{{\mathbf{D}}}

\newcommand{\bC}{{\mathbf{C}}}
\newcommand{\bH}{{\mathbf{H}}}
\newcommand{\bJ}{{\mathbf{J}}}

\newcommand{\fp}{{\mathfrak{p}}}

\newcommand{\fC}{{\mathfrak{C}}}

\newcommand{\cA}{{\mathcal{A}}}

\newcommand{\cE}{{\mathcal{E}}}

\newcommand{\cD}{{\mathcal{D}}}
\newcommand{\cL}{{\mathcal{L}}}
\newcommand{\cR}{{\mathcal{R}}}
\newcommand{\cLR}{{\mathcal{LR}}}

\newcommand{\cO}{{\mathcal{O}}}

\newcommand\Irr{\operatorname{Irr}}

\newtheorem{thm}{Theorem}[section]
\newtheorem{cor}[thm]{Corollary}
\newtheorem{prop}[thm]{Proposition}
\newtheorem{lem}[thm]{Lemma}

\theoremstyle{definition}
\newtheorem{defn}[thm]{Definition}
\newtheorem{exmp}[thm]{Example}

\theoremstyle{remark}
\newtheorem{rem}[thm]{Remark}

\renewcommand{\leq}{\leqslant}
\renewcommand{\geq}{\geqslant}
\renewcommand{\atop}[2]{\genfrac{}{}{0pt}{}{#1}{#2}}

\address{Department of Mathematical Sciences, King's College, 
Aberdeen AB24 3UE, Scotland, UK}

\email{geck@maths.abdn.ac.uk}

\begin{document}

\date{}

\title{On Iwahori--Hecke algebras with unequal parameters and
Lusztig's isomorphism theorem}

\author{Meinolf Geck}

\subjclass[2000]{Primary 20C08; Secondary 20G40}

\begin{abstract} 
By Tits' deformation argument, a generic Iwahori--Hecke algebra $\bH$ 
associated to a finite Coxeter group $W$ is abstractly isomorphic to the 
group algebra of $W$. Lusztig has shown how one can construct an explicit 
isomorphism, provided that the Kazhdan--Lusztig basis of $\bH$ satisfies 
certain deep properties. If $W$ is crystallographic and $\bH$ is a 
one-parameter algebra, then these properties are known to hold thanks to a 
geometric interpretation. In this paper, we develop some new general methods 
for verifying these properties, and we do verify them for two-parameter 
algebras of type $I_2(m)$ and $F_4$ (where no geometric interpretation is
available in general). Combined with previous work by Alvis, Bonnaf\'e, 
DuCloux, Iancu and the author, we can then extend Lusztig's construction 
of an explicit isomorphism to all types of $W$, without any restriction on 
the parameters of $\bH$.
\end{abstract}

\maketitle

\begin{center}
{\it Dedicated to Professor Jacques Tits on his 80th birthday}
\end{center}
\bigskip

\pagestyle{myheadings}
\markboth{Geck}{On Iwahori--Hecke algebras with unequal parameters}

\section{Introduction} \label{sec0}

Let $(W,S)$ be a Coxeter system where $W$ is finite. Let $F$ be a field of
characteristic zero and $A=F[v_s^{\pm 1} \mid s \in S]$ the ring of Laurent 
polynomials over $F$, where $\{v_s\mid s \in S\}$ is a collection of 
indeterminates such that $v_s=v_t$ whenever $s,t \in S$ are conjugate in 
$W$. Let $\bH$ be the associated ``generic'' Iwahori--Hecke algebra. This 
is an associative algebra over $A$, which is free as an $A$-module with 
basis $\{T_w\mid w \in W\}$. The multiplication is given by the rule 
\[ T_sT_w=\left\{\begin{array}{cl} T_{sw} & \quad \mbox{if $l(sw)>l(w)$},\\
T_{sw}+(v_s-v_s^{-1})T_w & \quad \mbox{if $l(sw)<l(w)$},\end{array}
\right.\]
where $s\in S$ and $w\in W$; here, $l\colon W \rightarrow \Z_{\geq 0}$ is 
the usual length function on $W$. 

Let $K$ be the field of fractions of $A$. By scalar extension, we obtain a 
$K$-algebra $\bH_K=K \otimes_A \bH$, which is well-known to be separable.
On the other hand, there is a unique ring homomorphism $\theta_1\colon A 
\rightarrow F$ such that $\theta_1(v_s)=1$ for all $s \in S$. Then we can 
regard $F$ as an $A$-algebra (via $\theta_1$) and obtain $F \otimes_A \bH=
F[W]$, the group algebra of $W$ over $F$. By a general deformation 
argument due to Tits (see \cite[Chap.~IV, \S 2, Exercise~27]{bour}), one 
can show that $\bH_{K'}$ and $K'[W]$ are abstractly isomorphic where 
$K'\supseteq K$ is a sufficiently large field extension. 

One of the purposes of this paper is to prove the following 
finer result which was first obtained by Lusztig \cite{Lu0} for finite 
Weyl groups in the case where all $v_s$ ($s \in S$) are equal.

\begin{thm} \label{Mmain} There exists an algebra homomorphism $\psi 
\colon \bH\rightarrow A[W]$ with the following properties:
\begin{itemize}
\item[(a)] If we extend scalars from $A$ to $F$ (via $\theta_1$), then 
$\psi$ induces the identity map.
\item[(b)] If we extend scalars from $A$ to $K$, we obtain an isomorphism 
$\psi_K\colon \bH_K \stackrel{\sim}{\rightarrow} K[W]$. 
\end{itemize}
\end{thm}

In particular, (b) implies that, if $F$ is a splitting field for $W$, then 
$\bH_K\cong K[W]$ is a split semisimple algebra. Recall that it is known that
$F_0={\Q}\bigl(\cos(2\pi/m_{st}\bigr) \mid s,t \in S) \subseteq \R$ is a 
splitting field for $W$; see \cite[Theorem~6.3.8]{gepf}. (Here, $m_{st}$
denotes the order of $st$ in $W$.) Note that $F_0=\Q$ if $W$ is a finite Weyl 
group, that is, if $m_{st} \in \{2,3, 4,6\}$ for all $s,t \in S$. 

The above result shows that, when $W$ is finite, the algebra $\bH_K$ and
its representation theory can be understood, at least in principle, via
the isomorphism $\bH_K \stackrel{\sim}{\rightarrow} K[W]$; see \cite{gepf}
and \cite[\S 20--24]{Lusztig03} where this is further developped. 

This paper is organised as follows. In Section~2, we recall the basic facts
about Kazhdan--Lusztig bases and cells. We present Lusztig's conjectures 
{\bf P1}--{\bf P15} and explain, following \cite{Lusztig03}, how the 
validity of these conjectures leads to a proof of Theorem~\ref{Mmain}. In 
this argument, a special role is played by Lusztig's asymptotic ring $\bJ$ 
which is defined using the leading coefficients of the structure constants
of the Kazhdan--Lusztig basis. 

Now, {\bf P1}--{\bf P15} are known to hold for finite Weyl groups in the 
equal parameter case, thanks to a deep geometric interpretation of the 
Kazhdan--Lusztig basis; see Kazhdan--Lusztig \cite{KaLu2}, Lusztig \cite{Lusztig03}, Springer \cite{Spr}. The case of non-crystallographic finite Coxeter 
groups is covered by Alvis \cite{Alvis87} and DuCloux \cite{Fokko}. So it 
remains to consider the case of unequal parameters where $W$ is of type 
$B_n$, $F_4$ or $I_2(m)$ ($m$ even). Type $B_n$ (with two independent 
parameters and a certain monomial order on them) has been dealt with by 
Bonnaf\'e, Iancu and the author; see \cite{BI}, \cite{BI2}, \cite{geia06}, 
\cite{myrel06}. In Sections~3 and~4, we develop new general methods for 
verifying {\bf P1}--{\bf P15}, based on the ``leading matrix coefficients'' 
introduced in \cite{my02}. In Section~5, we show how this can be used to 
deal with $W$ of type $F_4$ and $I_2(m)$, for all choices of parameters. 
We also indicate how our methods lead to a new proof of {\bf P1}--{\bf P15} 
for type $H_4$, which is based on the results of Alvis \cite{Alvis87} and 
Alvis--Lusztig \cite{AlLu82} but which does not rely on DuCloux's 
computation \cite{Fokko} of all structure constants of the 
Kazhdan--Lusztig basis.

Finally, we put all the pieces into place and complete the proof of 
Theorem~\ref{Mmain}.

\section{The Kazhdan--Lusztig basis} \label{sec1}

It will be convenient to slightly change the setting of the introduction.
So let $(W,S)$ be a Coxeter system and $l\colon W \rightarrow \Z_{\geq 0}$ 
be the usual length function. Throughout this paper, $W$ will be finite.
Let $\Gamma$ be an abelian group (written additively). Following Lusztig 
\cite{Lusztig03}, a function $L \colon W \rightarrow \Gamma$ is called a 
{\em weight function} if $L(ww')=L(w)+L(w')$ whenever $w,w'\in W$ are such 
that $l(ww')=l(w)+l(w')$. Note that $L$ is uniquely determined by the 
values $\{L(s)\mid s \in S\}$. Furthermore, if $\{c_s \mid s \in S\}$ is 
a collection of elements in $\Gamma$ such that $c_s=c_t$ whenever $s,t 
\in S$ are conjugate in $W$, then there is (unique) weight function 
$L\colon W \rightarrow \Gamma$ such that $L(s)=c_s$ for all $s \in S$. 

Let $R\subseteq \C$ be a subring and $A=R[\Gamma]$ be the free $R$-module 
with basis $\{\varepsilon^g \mid g\in \Gamma\}$. There is a well-defined 
ring structure on $A$ such that $\varepsilon^g\varepsilon^{g'}=
\varepsilon^{g+g'}$ for all $g,g' \in \Gamma$. We write $1=\varepsilon^0 
\in A$. Given $a\in A$ we denote by $a_g$ the coefficient of 
$\varepsilon^g$, so that $a=\sum_{g\in \Gamma} a_g\,\varepsilon^g$.
Let $\bH=\bH_A(W,S,L)$ be the {\em generic Iwahori--Hecke algebra} over $A$
with parameters $\{v_s \mid s\in S\}$ where $v_s:=\varepsilon^{L(s)}$ for
$s\in S$. This an associative algebra which is free as an $A$-module, with
basis $\{T_w\mid w \in W\}$. The multiplication is given by the rule
\[ T_sT_w=\left\{\begin{array}{cl} T_{sw} & \quad \mbox{if $l(sw)>l(w)$},\\
T_{sw}+(v_s-v_s^{-1})T_w & \quad \mbox{if $l(sw)<l(w)$},\end{array}
\right.\]
where $s\in S$ and $w\in W$. The element $T_1$ is the identity element.

\begin{exmp} \label{Mweightint} Assume that $\Gamma=\Z$. Then $A$ is
nothing but the ring of Laurent polynomials over $R$ in an
indeterminate~$\varepsilon$; we will usually denote $v=\varepsilon$. 
Then $\bH$ is an associative algebra over $A=R[v,v^{-1}]$ with
relations:
\[ T_sT_w=\left\{\begin{array}{cl} T_{sw} & \quad \mbox{if $l(sw)>l(w)$},\\
T_{sw}+(v^{c_s}-v^{-c_s})T_w & \quad \mbox{if $l(sw)<l(w)$},\end{array}
\right.\]
where $s\in S$ and $w\in W$.  
This is the setting of Lusztig \cite{Lusztig03}.
\end{exmp}

\begin{exmp} \label{Masym} (a) Assume that $\Gamma=\Z$ and $L$ is constant 
on $S$; this case will be referred to as the {\em equal parameter case}. 
Note that we are automatically in this case when $W$ is of type $A_{n-1}$, 
$D_n$, $I_2(m)$ where $m$ is odd, $H_3$, $H_4$, $E_6$, $E_7$ or $E_8$ (since
all generators in $S$ are conjugate in $W$).

(b) Assume that $W$ is finite and irreducible. Then unequal parameters
can only arise in types $B_n$, $I_2(m)$ where $m$ is even, and $F_4$.
\end{exmp}

\begin{exmp} \label{Mrem12}
A ``universal'' weight function is given as follows. Let $\Gamma_0$ be the
group of all tuples $(n_s)_{s \in S}$ where $n_s \in \Z$ for all $s \in S$
and $n_s=n_t$ whenever $s,t\in S$ are conjugate in $W$.  (The addition is
defined componentwise). Let $L_0\colon W \rightarrow \Gamma_0$ be the
weight function given by sending $s\in S$ to the tuple $(n_t)_{t \in S}$
where $n_t=1$ if $t$ is conjugate to $s$ and $n_t=0$, otherwise. Let 
$A_0=R[\Gamma_0]$ and $\bH_0=\bH_{A_0}(W,S,L_0)$ be the associated 
Iwahori--Hecke algebra, with parameters $\{v_s\mid s \in S\}$. Then $A_0=
R[\Gamma_0]$ is nothing but the ring of Laurent polynomials in 
indeterminates $v_s$ ($s\in S$) with coefficients in $R$, where $v_s=v_t$ 
whenever $s,t\in S$ are conjugate in $W$. Furthermore, if $S'\subseteq S$ 
is a set of representatives for the classes of $S$ under conjugation, then 
$\{v_s \mid s \in S'\}$ are algebraically independent.
\end{exmp}

\begin{rem} \label{univers1}
Let $k$ be any commutative ring (with $1$) and assume we are given
a collection of elements $\{\xi_s \mid s \in S\} \subseteq k^\times$ such
that $\xi_s=\xi_t$ whenever $s,t \in S$ are conjugate in $W$. Then we
have an associated Iwahori--Hecke algebra $H=H_k(W,S,\{\xi_s\})$ over $k$.
Again, this is an associative algebra; it is free as a $k$-module with
basis $\{T_w\mid w \in W\}$. The multiplication is given by the rule
\[ T_sT_w=\left\{\begin{array}{cl} T_{sw} & \quad \mbox{if $l(sw)>l(w)$},\\
T_{sw}+(\xi_s-\xi_s^{-1})T_w & \quad \mbox{if $l(sw)<l(w)$},\end{array}
\right.\]
where $s\in S$ and $w\in W$. Now let $A_0$ be as in Example~\ref{Mrem12},
 where $R=\Z$. Then we can certainly find a (unique) unital ring 
homomorphism $\theta_0 \colon A_0\rightarrow k$ such that $\theta_0(v_s)= 
\xi_s$ for all $s \in S$. Regarding $k$ as an $A_0$-module (via $\theta_0$),
we find that $H$ is obtained by extension of scalars from $\bH_0$:
\[ H_k(W,S,\{\xi_s\})\cong k \otimes_A \bH_0.\]
We conclude that $H_k(W,S,\{\xi_s\})$ can always be obtained by 
``specialisation'' from the ``universal'' generic Iwahori--Hecke algebra 
$\bH_0$.
\end{rem}

We now recall the basic facts about the Kazhdan--Lusztig basis of
$\bH$, following Lusztig \cite{Lusztig83}, \cite{Lusztig03}. For this
purpose, we need to assume that $\Gamma$ admits a total ordering $\leq$
which is compatible with the group structure, that is, whenever $g,g',h
\in \Gamma$ are such that $g\leq g'$, then $g+h\leq g'+h$. Such an order 
on $\Gamma$ will be called a {\em monomial order}. One readily checks that 
this implies that $A=R[\Gamma]$ is an integral domain; we usually reserve 
the letter $K$ to denote its field of fractions. We will assume throughout 
that 
\[ L(s)>0 \quad \mbox{for all $s \in S$}.\]
Now, there is a unique ring involution $A\rightarrow A$, $a \mapsto 
\bar{a}$, such that $\overline{\varepsilon^g}=\varepsilon^{-g}$ for all 
$g\in\Gamma$. We can extend this map to a ring involution $\bH \rightarrow 
\bH$, $h \mapsto \overline{h}$, such that
\[ \overline{\sum_{w \in W} a_w T_w}=\sum_{w \in W} \bar{a}_w
T_{w^{-1}}^{-1} \qquad (a_w \in A).\]
We define $\Gamma_{\geq 0}=\{g\in \Gamma\mid g\geq 0\}$ and denote by
$\Z[\Gamma_{\geq 0}]$ the set of all integral linear combinations of
terms $\varepsilon^g$ where $g\geq 0$. The notations $\Z[\Gamma_{>0}]$,
$\Z[\Gamma_{\leq 0}]$, $\Z[\Gamma_{<0}]$ have a similar meaning.

\begin{thm}[Kazhdan--Lusztig \protect{\cite{KaLu}}, Lusztig
\protect{\cite{Lusztig83}, \cite{Lusztig03}}] \label{Mklbase} For each $w
\in W$, there exists a unique $C_w'\in \bH$ (depending on $\leq$) such that
\begin{itemize}
\item $\overline{C}_w'=C_w'$ and
\item $C_w'=T_w+\sum_{y \in W} p_{y,w} T_y$ where $p_{y,w}\in 
{\Z}[\Gamma_{<0}]$ for all $y \in W$.
\end{itemize}
The elements $\{C_w'\mid w\in W\}$ form an $A$-basis of $\bH$, and we have
$p_{y,w}=0$ unless $y<w$ (where $<$ denotes the Bruhat--Chevalley order on 
$W$).
\end{thm}

Here we follow the original notation in \cite{KaLu}, \cite{Lusztig83}; the
element $C_w'$ is denoted by $c_w$ in \cite[Theorem~5.2]{Lusztig03}. As in
\cite{Lusztig03}, it will be convenient to work with the following
alternative version of the Kazhdan--Lusztig basis. We set $\bC_w=
(C_w')^{\dagger}$ where $\dagger\colon \bH \rightarrow 
\bH$ is the $A$-algebra automorphism defined by $T_s^\dagger=-T_s^{-1}$ 
($s \in S$); see \cite[3.5]{Lusztig03}. Note that $\overline{h} =
j(h)^\dagger=j(h^\dagger)$ for all $h \in \bH$ where $j \colon \bH
\rightarrow\bH$ is the ring involution such that $j(a)=\bar{a}$ for
$a \in A$ and $j(T_w)=(-1)^{l(w)}T_w$ for $w \in W$. Thus, we have
\begin{itemize}
\item $\overline{\bC}_w=j(C_w')=\bC_w$ and
\item $\bC_w=(-1)^{l(w)}T_w +
\sum_{y \in W} (-1)^{l(y)}\overline{p}_{y,w} T_y$ where $\overline{p}_{y,w}
\in {\Z}[\Gamma_{>0}]$.
\end{itemize}
Since the elements $\{\bC_w\mid w\in W\}$ form a basis of $\bH$, we can write
\[ \bC_x \bC_y=\sum_{z \in W} h_{x,y,z}\, \bC_z\qquad \mbox{for any
$x,y \in W$},\]
where $h_{x,y,z}=\overline{h}_{x,y,z} \in A$ for all $x,y,z\in W$. 
The structure constants $h_{x,y,z}$ can de described more explicitly in 
the following special case. Let $s \in S$ and $w \in W$. Then we have
\[ \renewcommand{\arraystretch}{1.2}
\bC_s\,\bC_w = \left\{\begin{array}{ll} \displaystyle{\bC_{sw}+
\sum_{\atop{y \in W}{sy<y<w}} \mu_{y,w}^s \, \bC_y} &\quad \mbox{if $sw>w$},
\\(v_s+v_s^{-1})\,\bC_w &\quad \mbox{if $sw<w$}, \end{array}\right.\]
where $\mu_{y,w}^s \in A$; see \cite[Theorem~6.6]{Lusztig03}.

\begin{rem} \label{Mcellrep} We refer to \cite[Chap.~8]{Lusztig03} for the 
definition of the preorders $\leq_{\cL}$, $\leq_{\cR}$, $\leq_{\cLR}$ and 
the corresponding equivalence relations $\sim_{\cL}$, $\sim_{\cR}$, 
$\sim_{\cLR}$ on $W$. (Note that these depend on the weight function $L$ 
and the monomial order on $\Gamma$.) The equivalence classes with respect 
to these relations are called left, right and two-sided cells of $W$, 
respectively.

Each left cell $\fC$ gives rise to a representation of $\bH$ (and of $W$).
This is constructed as follows (see \cite[\S 7]{Lusztig83}). Let $[\fC]_A$
be an $A$-module with a free $A$-basis $\{e_w \mid w \in \fC\}$. Then the
action of $\bC_w$ ($w \in W$) on $[\fC]_A$ is given by the Kazhdan--Lusztig
structure constants, that is, we have
\[ \bC_w.e_x = \sum_{y\in \fC} h_{w,x,y} \, e_y \qquad \mbox{for all
$x\in \fC$ and $w\in W$}.\]
Furthermore, let $\theta_1 \colon A\rightarrow R$ be the unique ring
homomorphism such that $\theta_1(\varepsilon^g)=1$ for all $g \in \Gamma$.
Extending scalars from $A$ to $R$ (via $\theta_1$), we obtain a module
$[\fC]_1:=R \otimes_A [\fC]_A$ for $R[W]=R\otimes_A \bH$.
\end{rem}

Following Lusztig \cite{Lusztig03}, given $z \in W$, we define
\[ \ba(z):= \min \{g\in\Gamma_{\geq 0}\mid \varepsilon^g\,h_{x,y,z} \in
\Z[\Gamma_{\geq 0}] \mbox{ for all $x,y\in W$}\}.\]
Thus, we obtain a function $\ba\colon W \rightarrow \Gamma$. (If 
$\Gamma=\Z$ with its natural order, then this reduces to the function
first defined by Lusztig \cite{Lu1}.) Given $x,y,z\in W$, we define 
$\gamma_{x,y,z^{-1}} \in \Z$ to be the constant term of 
$\varepsilon^{\ba(z)}\, h_{x,y,z}$, that is, we have
\[ \varepsilon^{\ba(z)}\, h_{x,y,z} \equiv \gamma_{x,y,z^{-1}} 
\quad \bmod {\Z} [\Gamma_{>0}].\]
Next, recall that $p_{1,z}$ is the coefficient of $T_1$ in the expansion
of $C_w'$ in the $T$-basis. By \cite[Prop.~5.4]{Lusztig03}, we have 
$p_{1,z} \neq 0$.  As in  \cite[14.1]{Lusztig03}, we define $\Delta(z)
\in \Gamma_{\geq 0}$ and $0\neq n_z \in \Z$ by the condition that
$\varepsilon^{\Delta(z)}p_{1,z} \equiv n_z \bmod {\Z}[\Gamma_{<0}]$. We set
\[ \cD=\{z \in W \mid \ba(z)=\Delta(z)\}.\]
Now Lusztig \cite[Chap.~14]{Lusztig03} has formulated the following
$15$ conjectures:
\begin{itemize}
\item[\bf P1.] For any $z\in W$ we have $\ba(z)\leq \Delta(z)$.
\item[\bf P2.] If $d \in \cD$ and $x,y\in W$ satisfy $\gamma_{x,y,d}\neq 0$,
then $x=y^{-1}$.
\item[\bf P3.] If $y\in W$, there exists a unique $d\in \cD$ such that
$\gamma_{y^{-1},y,d}\neq 0$.
\item[\bf P4.] If $z'\leq_{\cL\cR} z$ then $\ba(z')\geq \ba(z)$. Hence, if
$z'\sim_{\cL\cR} z$, then $\ba(z)=\ba(z')$.
\item[\bf P5.] If $d\in \cD$, $y\in W$, $\gamma_{y^{-1},y,d}\neq 0$, then
$\gamma_{y^{-1},y,d}=n_d=\pm 1$.
\item[\bf P6.] If $d\in \cD$, then $d^2=1$.
\item[\bf P7.] For any $x,y,z\in W$, we have $\gamma_{x,y,z}=\gamma_{y,z,x}$.
\item[\bf P8.] Let $x,y,z\in W$ be such that $\gamma_{x,y,z}\neq 0$. Then
$x\sim_{\cL} y^{-1}$, $y \sim_{\cL} z^{-1}$, $z\sim_{\cL} x^{-1}$.
\item[\bf P9.] If $z'\leq_{\cL} z$ and $\ba(z')=\ba(z)$, then $z'\sim_{\cL}z$.\item[\bf P10.] If $z'\leq_{\cR} z$ and $\ba(z')=\ba(z)$, then $z'\sim_{\cR}z$.
\item[\bf P11.] If $z'\leq_{\cL\cR} z$ and $\ba(z')=\ba(z)$, then
$z'\sim_{\cL\cR}z$.
\item[\bf P12.] Let $I\subseteq S$ and $W_I$ be the parabolic subgroup
generated by $I$. If $y\in W_I$, then $\ba(y)$ computed in terms of $W_I$
is equal to $\ba(y)$ computed in terms of~$W$.
\item[\bf P13.] Any left cell $\fC$ of $W$ contains a unique element
$d\in \cD$. We have $\gamma_{x^{-1},x,d}\neq 0$ for all $x\in \fC$.
\item[\bf P14.] For any $z\in W$, we have $z \sim_{\cL\cR} z^{-1}$.
\item[\bf P15.] If $x,x',y,w\in W$ are such that $\ba(w)=\ba(y)$, then
\[\sum_{y' \in W} h_{w,x',y'}\otimes h_{x,y',y}=
\sum_{y'\in W} h_{y',x',y} \otimes h_{x,w,y'}\quad \mbox{in
${\Z}[\Gamma] \otimes_{\Z} {\Z}[\Gamma]$}. \]
\end{itemize}
(The above formulation  of {\bf P15} is taken from Bonnaf\'e \cite{BI2}.)

\begin{rem} \label{equalgeo} Assume that we are in the equal parameter 
case; see Example~\ref{Masym}. In this case, $A=\Z[\Gamma]$ is nothing but 
the ring of Laurent polynomials in one variable~$v$. Suppose that all 
polynomials $p_{x,y}\in \Z[v^{-1}]$ and all structure constants $h_{x,y,z}
\in \Z[v,v^{-1}]$ have non-negative coefficients. Then Lusztig 
\cite[Chap.~15]{Lusztig03} shows that {\bf P1}--{\bf P15} follow. 

Now, if $(W,S)$ is a finite Weyl group, that is, if $m_{st}\in\{2,3,4,6\}$
for all $s,t \in S$, then the required non-negativity of the coefficients
is shown by using a deep geometric interpretation of the Kazhdan--Lusztig
basis; see Kazhdan--Lusztig \cite{KaLu2}, Springer \cite{Spr}. Thus,
{\bf P1}--{\bf P15} hold for finite Weyl groups in the equal parameter case.
If $(W,S)$ is of type $I_2(m)$ (where $m\not\in\{2,3,4,6\}$), $H_3$ or
$H_4$, the non-negativity of the coefficients has been checked
explicitly by Alvis \cite{Alvis87} and DuCloux \cite{Fokko}.

Note that simple examples show that the coefficients of the polynomials
$p_{y,w}$ or $h_{x,y,z}$ may be negative in the presence of unequal 
parameters; see Lusztig \cite[p.~106]{Lusztig83}, \cite[\S 7]{Lusztig03}. 
\end{rem}

We now use {\bf P1}--{\bf P15} to perform the following constructions,
following Lusztig \cite{Lusztig03}. Let $\bJ$ be the free ${\Z}$-module 
with basis $\{t_w\mid w\in W\}$. We define a bilinear product on $\bJ$ by
\[ t_xt_y=\sum_{z\in W} \gamma_{x,y,z^{-1}}\, t_z \qquad (x,y\in W).\]

\begin{rem} \label{MJinvers} By \cite[5.6]{Lusztig03}, the map 
$\bH \rightarrow \bH$ defined by $\bC_w\mapsto \bC_{w^{-1}}$ ($w \in W$)
is an anti-involution; so we have $h_{x,y,z}=h_{y^{-1},x^{-1},z^{-1}}$ for
all $w,x,y,z \in W$. In particular, this implies that $\ba(z)=\ba(z^{-1})$
for all $z \in W$. By \cite[13.9]{Lusztig03}, the map $\bJ\rightarrow \bJ$ 
defined by $t_w \mapsto t_{w^{-1}}$ ($w\in W$) also is an anti-involution 
of $\bJ$; so we have $\gamma_{x,y,z}=\gamma_{y^{-1},x^{-1},z^{-1}}$ for 
all $x,y,z \in W$.
\end{rem}

\begin{thm}[Lusztig \protect{\cite[Chap.~18]{Lusztig03}}] \label{MJasymp}
Assume that {\bf P1--P15} hold. Then $\bJ$ is an associative ring with
identity element $1_{\bJ}=\sum_{d\in \cD} n_dt_d$. Let $\bJ_A=A
\otimes_{\Z} \bJ$. Then we have a unital homomorphism of $A$-algebras
\[\phi \colon \bH \rightarrow \bJ_A, \qquad \bC_w \mapsto
\sum_{\atop{z\in W,d\in \cD}{\ba(z)=\ba(d)}} h_{w,d,z}\,n_d\,t_z,\]
\end{thm}

The ring $\bJ$ will be called the {\em asymptotic algebra} associated to
$\bH$ (with respect to $\leq$). It first appeared in \cite{Lu2} in the
equal parameter case.

\begin{rem} \label{note1} In \cite[Theorem~18.9]{Lusztig03}, the formula for
$\phi$ looks somewhat different: instead of the factor $n_d$, there
is a factor $\hat{n}_z$ which is defined as follows. Given $z\in W$, there
is a unique element of $\cD$ such that $\gamma_{z,z^{-1},d}\neq 0$; then
$\hat{n}_z=n_d=\pm 1$ (see {\bf P3}, {\bf P5}, {\bf P13}). Now one easily 
checks, using {\bf P1}--{\bf P15}, that the map $t_w \mapsto \hat{n}_w 
\hat{n}_{w^{-1}}t_w$ defines a ring involution of $\bJ$. Composing 
Lusztig's homomorphism in \cite[18.9]{Lusztig03} with this involution, we 
obtain the above formula (which seems more natural; see, e.g., the
discussion in \cite[\S 5]{myedin08}).
\end{rem}

The structure of $\bJ$ is to some extent clarified by the following remark,
which is taken from \cite[20.1]{Lusztig03}.

\begin{rem} \label{MJcomp1} Assume that {\bf P1}--{\bf P15} hold. Recall
that $A=R[\Gamma]$ where $R \subseteq \C$ is a subring. Now assume that $R$
is a field. Let $\theta_1 \colon A\rightarrow R$ be the unique ring
homomorphism such that $\theta_1(\varepsilon^g)=1$ for all $g \in \Gamma$.
Then $R \otimes_A \bH=R[W]$. Via $\theta$ and extension of scalars, we 
obtain an induced homomorphism of $R$-algebras
\[ \phi_1 \colon R[W] \rightarrow \bJ_R=R \otimes_{\Z} J, \qquad
\bC_w\mapsto \sum_{\atop{z\in W,d\in \cD}{\ba(z)=\ba(d)}}
\theta(h_{w,d,z})\,n_d\,t_z.\]
Now, the kernel of $\phi_1$ is a nilpotent ideal in $R[W]$; see 
\cite[Prop.~18.12(a)]{Lusztig03}. Since $R[W]$ is a semisimple algebra, 
we conclude that $\phi_1$ is injective and, hence, an isomorphism.  In 
particular, we can now conclude that
\begin{itemize}
\item {\em $\bJ_R\cong R[W]$ is a semisimple algebra};
\item {\em $\bJ_R$ is split if $R$ is a splitting field for $W$}.
\end{itemize}
We can push this discussion even further. Let $P$ be the matrix of
$\phi \colon \bH \rightarrow \bJ_A$ with respect to the standard bases
of $\bH$ and $\bJ_A$. Let $P_1$ be the matrix obtained by applying $\theta_1$
to all entries of $P$. Then $P_1$ is the matrix of $\phi_1$ with
respect to the standard bases of $R[W]$ and $\bJ_R$. We have seen above
that $\det(P_1)\neq 0$. Hence, clearly, we also have $\det(P)\neq 0$.
Consequently, we obtain an induced isomorphism  $\phi_K \colon \bH_K 
\stackrel{\sim}{\rightarrow} \bJ_K$ where $K$ is the field of fractions 
of $A$. In particular, if $R$ is a splitting field for $W$, then $\bJ_R$ 
is split semisimple and, hence, $\bH_K \cong \bJ_K$ will be split
semisimple, too.
\end{rem}

We now obtain the following result which was first obtained by Lusztig 
\cite{Lu0} (for finite Weyl groups in the equal parameter case).

\begin{thm}[Lusztig] \label{MisoH} Assume that $R$ is a field and that 
{\bf P1}--{\bf P15} hold. Then there exists an algebra homomorphism 
$\psi\colon \bH \rightarrow A[W]$ with the following properties:
\begin{itemize}
\item[(a)] Let $\theta_1\colon A\rightarrow R$ be the unique ring
homomorphism such that $\theta_1(\varepsilon^g)=1$ for all $g \in \Gamma$. 
If we extend scalars from $A$ to $R$ (via $\theta_1$), then $\psi$ induces
the identity map.
\item[(b)] If we extend scalars from $A$ to $K$ (the field of fractions
of $A$), then $\psi$ induces an isomorphism $\psi_K\colon \bH_K 
\stackrel{\sim}{\rightarrow} K[W]$. In particular, $\bH_K$ is a semisimple
algebra, which is split if $R$ is a splitting field for~$W$.
\end{itemize}
\end{thm}

\begin{proof} As in Remark~\ref{MJcomp1}, we have an isomorphism $\phi_1
\colon R[W] \stackrel{\sim}{\rightarrow}\bJ_R$. Let $\alpha:=\phi_1^{-1}
\colon \bJ_R \stackrel{\sim}{\rightarrow} R[W]$. By extension of scalars,
we obtain an isomorphism of $A$-algebras $\alpha_A\colon \bJ_A
\stackrel{\sim}{\rightarrow} A[W]$. Now set $\psi:=\alpha_A \circ
\phi\colon \bH \rightarrow A[W]$.

(a) If we extend scalars from $A$ to $R$ via $\theta_1$, then $\bH_R=
R[W]$. Furthermore, $\phi \colon \bH \rightarrow \bJ_A$ induces the map
$\phi_1$ already considered at the beginning of the proof. Hence
$\psi$ induces the identity map.

(b) This immediately follows from (a) by a formal argument: Let $Q$ be
the matrix of the $A$-linear map $\psi$ with respect to the standard
$A$-bases of $\bH$ and $A[W]$. We only need to show that $\det(Q) \neq 0$.
But, by (a), we have $\theta_1(\det(Q))=1$; in particular, $\det(Q)\neq 0$.

Finally, note that, if $R$ is a splitting field for $W$, then so is $K$.
Hence, in this case, $\bH_K \cong K[W]$ is a split semisimple algebra.
\end{proof}

Note that the {\em statement} of the above result does not make any
reference to the monomial order $\leq$ on $\Gamma$ or the corresponding
Kazhdan--Lusztig basis; these are only needed in the proof. 

\begin{rem} \label{cellJ} Assume that {\bf P1}--{\bf P14} hold. Then the 
partitions of $W$ into left, right and two-sided cells can be recovered
from the structure of $\bJ$. Indeed, given $x,y \in W$, write
$x\leftrightarrow_{\cL} y$ if there exists some $z \in W$ such that
$\gamma_{x,y^{-1},z} \neq 0$. Then one easily checks that $\sim_{\cL}$
is the transitive closure of $\leftrightarrow_{\cL}$. (Note that, by 
\cite[Prop.~18.4(a)]{Lusztig03}, the relations $\sim_{\cL}$ and 
$\leftrightarrow_{\cL}$ are actually the same when we are in the equal 
parameter case.) Thus, the left cells are determined by $\bJ$. Furthermore, 
we have $x \sim_{\cR} y$ if and only if $x^{-1} \sim_{\cL} y^{-1}$. Finally,
by {\bf P4}, {\bf P9}, the two-sided cells are the smallest subsets of $W$ 
which are at the same time unions of left cells and unions of right cells.
\end{rem}

\section{The $\ba$-function and orthogonal representations} \label{secort1}
The aim of this and the following section is to develop some new methods 
for verifying {\bf P1}--{\bf P15} for a given group $W$ and weight 
function $L$. These methods should not rely on any positivity properties 
or geometric interpretations as mentioned in Remark~\ref{equalgeo}, so that 
we may hope to be able to apply them in the general case of unequal 
parameters. 

One of the main problems in the verification of {\bf P1}--{\bf P15} is the 
determination of the $\ba$-function. Note that, if we just wanted to use 
the definition of $\ba(z)$, then we would have to compute all structure 
constants $h_{x,y,z}$ where $x,y \in W$---which is very hard to get a hold 
on. We shall now describe a situation in which this problem can be solved 
by a different approach, which is inspired by \cite[\S 4]{geia06}.

For the rest of this section, let us assume that $R=\R$. Then $R$ is a
splitting field for $W$; see \cite[6.3.8]{gepf}. The set of irreducible
representations of $W$ (up to isomorphism) will be denoted by
\[ \Irr(W)=\{E^\lambda \mid \lambda \in \Lambda\}\]
where $\Lambda$ is some finite indexing set and $E^\lambda$ is an
$R$-vectorspace with a given $R[W]$-module structure.  We shall also write
\[d_\lambda=\dim E^\lambda \qquad \mbox{for all $\lambda \in \Lambda$}.\]
Let $K$ be the field of fractions of $A$. By extension of scalars, we obtain 
a $K$-algebra $\bH_K=K\otimes_A \bH$.  This algebra is known to be split 
semisimple; see \cite[9.3.5]{gepf}. Furthermore, by Tits' Deformation 
Theorem, the irreducible representations of $\bH_K$ (up to isomorphism) are 
in bijection with the irreducible representations of $W$; see 
\cite[8.1.7]{gepf}.  Thus, we can write
\[ \Irr(\bH_K)=\{E^\lambda_\varepsilon \mid \lambda \in \Lambda\}.\]
The correspondence $E^\lambda \leftrightarrow E^\lambda_\varepsilon$ is 
uniquely determined by the following condition:
\[ \mbox{trace}\bigl(w,E^\lambda\bigr)=\theta_1\bigl(\mbox{trace}(T_w,
E^\lambda_\varepsilon)\bigr) \qquad \mbox{for all $w \in W$},\]
where $\theta_1 \colon A \rightarrow F$ is the unique ring homomorphism
such that $\theta_1(\varepsilon^g)=1$ for all $g \in \Gamma$. Note also 
that $\mbox{trace}\bigl(T_w,E^\lambda_\varepsilon\bigr) \in A$ for 
all $w\in W$. Note that all these statements can be proved without using
{\bf P1}--{\bf P15}.

The algebra $\bH$ is {\em symmetric}, with trace from $\tau \colon \bH 
\rightarrow A$ given by $\tau(T_1)=1$ and $\tau(T_w)=0$ for $1 \neq w 
\in  W$. The sets $\{T_w \mid w \in W\}$ and $\{T_{w^{-1}}\mid w \in W\}$ 
form a pair of dual bases. Hence we have the following orthogonality 
relations for the irreducible representations of $\bH_K$:
\[ \sum_{w \in W} \mbox{trace}\bigl(T_w,E^\lambda_\varepsilon\bigr)
\,\mbox{trace}\bigl(T_{w^{-1}},E_\varepsilon^\mu\bigr)=\left
\{\begin{array}{cl} d_\lambda\,\bc_\lambda & \quad \mbox{if $\lambda=\mu$},
\\ 0 & \quad \mbox{if $\lambda \neq\mu$};\end{array}\right.\]
see \cite[8.1.7]{gepf}. Here, $0 \neq \bc_\lambda \in A$ and,
following Lusztig, we can write
\[ \bc_\lambda=f_\lambda\, \varepsilon^{-2\ba_\lambda}+
\mbox{combination of terms $\varepsilon^g$ where $g>-2\ba_\lambda$},\]
where $\ba_\lambda \in \Gamma_{\geq 0}$ and $f_\lambda$ is a strictly 
positive real number; see \cite[9.4.7]{gepf}. These invariants are 
explicitly known for all types of $W$; see Lusztig \cite[Chap.~22]{Lusztig03}. 

We shall also need the basis which is dual to the Kazhdan--Lusztig basis.
Let $\{\bD_w \mid w \in W\}\subseteq \bH$ be such that $\tau(\bC_x\,
\bD_{y^{-1}})= \delta_{xy}$ for all $x,y\in W$. Then
\[ h_{x,y,z}=\tau(\bC_x \bC_y \bD_{z^{-1}}) \quad \mbox{for all $x,y,z \in
W$}.\]
One also shows that $\bD_w$ can be written as a sum of $(-1)^{l(w)}T_w$
and a ${\Z}[\Gamma_{>0}]$-linear combination of terms $T_y$ ($y \in W$);
see \cite[Chap.~10]{Lusztig03} or \cite[2.4]{my02}

We now recall the basic facts concerning the leading matrix coefficients
introduced in \cite{my02}.  Let us write 
\begin{align*}
A_{\geq 0}=\mbox{set of $R$-linear combinations of terms $\varepsilon^g$ 
where $g\geq 0$},\\
A_{>0}=\mbox{set of $R$-linear combinations of terms $\varepsilon^g$ where
$g>0$}.
\end{align*}
Note that $1+A_{>0}$ is multiplicatively closed. Furthermore, every 
element $x\in K$ can be written in the form
\[ x=r_x\,\varepsilon^{\gamma_x}\frac{1+p}{1+q}\qquad \mbox{where $r_x 
\in R$, $\gamma_x \in \Gamma$ and $p,q\in A_{>0}$};\]
note that, if $x\neq 0$, then $r_x$ and $\gamma_x$ indeed are
{\em uniquely determined} by $x$; if $x=0$, we have $r_0=0$ and we set
$\gamma_0:=+\infty$ by convention. We set
\[{\cO}:=\{x \in K \mid \gamma_x \geq 0\} \qquad \mbox{and}\qquad
{\fp}:=\{x \in K \mid \gamma_x >0\}.\]
Then it is easily verified that $\cO$ is a valuation ring in $K$, with
maximal ideal $\fp$. Note that we have
\[ \cO \cap A=A_{\geq 0} \qquad \mbox{and}\qquad \fp\cap A=A_{>0}.\]
We have a well-defined $R$-linear ring homomorphism $\cO \rightarrow R$
with kernel $\fp$. The image of $x\in \cO$ in $R$ is called the
{\em constant term} of $x$. Thus, the constant term of $x$ is $0$ if
$x\in \fp$; the constant term equals $r_x$ if $x\in \cO^\times$.

By \cite[Prop.~4.3]{my02}, each $E^\lambda_\varepsilon$ affords a so-called 
{\em orthogonal representation}. By \cite[Theorem~4.4 and 
Remark~4.5]{my02}, this implies that there exists a basis of
$E^\lambda_\varepsilon$ such that the corresponding matrix representation 
$\rho^\lambda \colon \bH_K\rightarrow M_{d_\lambda}(K)$ has the following 
properties. Let $\lambda \in \Lambda$ and $1\leq i,j \leq d_\lambda$. For 
any $h\in \bH_K$, we denote by $\rho^\lambda_{ij}(h)$ the $(i,j)$-entry of 
the matrix $\rho^\lambda(h)$. Then
\[ \varepsilon^{\ba_\lambda} \rho^\lambda_{ij}(T_w) \in \cO,  \qquad
\varepsilon^{\ba_\lambda} \rho^\lambda_{ij}(\bC_w) \in \cO,  \qquad
\varepsilon^{\ba_\lambda} \rho^\lambda_{ij}(\bD_w) \in \cO\]
for any $w\in W$ and
\[ (-1)^{l(w)}\varepsilon^{\ba_\lambda} \rho^\lambda_{ij}(T_w) \equiv 
\varepsilon^{\ba_\lambda} \rho^\lambda_{ij}(\bC_w) \equiv 
\varepsilon^{\ba_\lambda} \rho^\lambda_{ij}(\bD_w) \bmod\fp.\]
Hence, the above three elements of $\cO$ have the same constant term
which we denote by $c_{w,\lambda}^{ij}$. The constants 
$c_{w,\lambda}^{ij}\in R$ are called the {\em leading matrix coefficients}
of $\rho^\lambda$. Given $w \in W$, there exists some $\lambda \in 
\Lambda$ and $i,j\in \{1,\ldots,d_\lambda\}$ such that $c_{w,\lambda}^{ij}
\neq 0$. We use this fact to define the following relation.

\begin{defn} \label{leadcell} Let $\lambda \in \Lambda$ and $w \in W$.
We write $E^\lambda \leftrightsquigarrow_{L} w$ if $c_{w,\lambda}^{ij}
\neq 0$ for some $i,j\in \{1,\ldots,d_\lambda\}$. 

(This is in analogy to Lusztig \cite[20.2]{Lusztig03} or 
\cite[p.~139]{LuBook}; see Lemma~\ref{leadcell1} below.)
\end{defn}

One can show that ``$\leftrightsquigarrow_{L}$'' does not depend on the 
choice of the orthogonal representations $\rho^\lambda$ (see 
\cite[Remark~3.10]{myedin08}), but we don't need this here. For our
purposes, the characterisation of ``$\leftrightsquigarrow_{L}$'' given in
the following result will be sufficient.

Recall from Remark~\ref{Mcellrep} that every left cell $\fC$ of $W$ gives
rise to a left $R[W]$-module denoted by $[\fC]_1$.

\begin{lem} \label{leadcell1} Let $\lambda \in \Lambda$ and $\fC$ be a left
cell of $W$. Then $E^\lambda \leftrightsquigarrow_{L} w$ for some $w\in\fC$ 
if and only if $E^\lambda$ is a constituent of $[\fC]_1$. 
\end{lem}

\begin{proof} Let $i \in \{1,\ldots,d_\lambda\}$. The assertion
immediately follows from the identity
\[ \frac{1}{f_\lambda}\sum_{k=1}^{d_\lambda} \sum_{w\in \fC}
(c_{w,\lambda}^{ik})^2=\mbox{multiplicity of $E^\lambda$ in $[\fC]_1$}.\]
which was proved in \cite[Prop.~4.7]{my02}. 
\end{proof}

\begin{rem} \label{hyp1a} Let $w,w'\in W$ and $\lambda \in \Lambda$ be 
such that $E^\lambda \leftrightsquigarrow_L w$ and $E^\lambda 
\leftrightsquigarrow_{L} w'$. Let $\fC$, $\fC'$ be the left cells such 
that $w \in \fC$ and $w' \in \fC'$. By Lemma~\ref{leadcell1}, $E^\lambda$ 
is a constituent of both $[\fC]_1$ and $[\fC']_1$. Hence, $\mbox{Hom}_{W}
([\fC]_1,[\fC']_1)\neq 0$ and so $\fC,\fC'$ are contained in the same 
two-sided cell. In particular, $w\sim_{\cLR} w'$.

This argument also implies {\bf P14}, i.e., the assertion that 
$w \sim_{\cLR} w^{-1}$ for all $w \in W$. Indeed, choose $\lambda\in 
\Lambda$ such that $E^\lambda \leftrightsquigarrow_{L} w$, that is, 
$c_{w,\lambda}^{ij}\neq 0$ for some $i,j \in \{1,\ldots,d_{\lambda}\}$. By 
\cite[Theorem~4.4]{my02}, we also have $c_{w^{-1},\lambda}^{ji}=c_{w,
\lambda}^{ij}\neq 0$ and so $E^\lambda \leftrightsquigarrow_L w^{-1}$. 
Hence, the previous discussion shows that $w \sim_{\cLR} w^{-1}$, as claimed.

(This was first proved by Lusztig \cite[Lemma~5.2]{LuBook} in the equal 
parameter case. One can check that Lusztig's proof also carries over to the 
case of unequal parameters.)
\end{rem}

\begin{lem} \label{abound} Let $z \in W$ and $\lambda \in \Lambda$ be such
$E^\lambda \leftrightsquigarrow_{L} z$. Then $\ba(z)\geq \ba_\lambda$.
\end{lem}

(A similar result was proved in \cite[Prop.~4.1]{geia06}, but under
additional assumptions. See also Lusztig \cite[Prop.~6.4]{Lu1} where this 
result was obtained in the equal parameter case, based on the geometric 
interpretation which is available there.)

\begin{proof} We begin by considering the structure constant $h_{x,y,z}$
for $x,y\in W$. We have $h_{x,y,z}=\tau(\bC_x \bC_y\bD_{z^{-1}})$. Now, 
by the general theory of symmetric algebras (see \cite[Chap.~7]{gepf}), we
have 
\[ \tau(h)=\sum_{\lambda \in \Lambda} \bc_\lambda^{-1} \mbox{trace}
(h,E^\lambda)= \sum_{\lambda \in \Lambda} \bc_\lambda^{-1}\,
\mbox{trace}\bigl(\rho^\lambda(h) \bigr)=\sum_{\lambda \in \Lambda}
\sum_{1\leq i\leq d_\lambda} \bc_\lambda^{-1}\, \rho^\lambda_{ii}(h),\]
for any $h \in \bH$. Since $\rho^\lambda(\bC_x\bC_y\bD_{z^{-1}})=
\rho^\lambda(\bC_x)\rho^\lambda (\bC_y)\rho^\lambda(\bD_{z^{-1}})$, we obtain
\[ h_{x,y,z}=\sum_{\mu\in \Lambda} \sum_{1\leq i,j,k\leq d_\mu}
\bc_\mu^{-1}\, \rho^\mu_{ij}(\bC_x)\,\rho^\mu_{jk}(\bC_y)\,
\rho^\mu_{ki}(\bD_{z^{-1}}).\]
We multiply this identity on both sides by $\rho^\lambda_{ls}
(\bD_{x^{-1}})\, \rho^\lambda_{rl} (\bD_{y^{-1}})$ (where $\lambda \in
\Lambda$ and $1\leq l,r,s\leq d_\lambda$) and sum over all $x,y\in W$.
Now, since $\{\bC_w \mid w \in w\}$ and $\{\bD_{w^{-1}} \mid w \in W\}$
form a pair of dual bases for $\bH$, we have the following Schur relations
(see \cite[Chap.~7]{gepf}):
\begin{equation*}
\sum_{w \in W} \rho_{ij}^\lambda(\bC_w)\,\rho^\mu_{kl}(\bD_{w^{-1}})=
\delta_{il}\delta_{jk} \delta_{\lambda\mu} \bc_\lambda,
\end{equation*}
where $\lambda,\mu \in \Lambda$, $1\leq i,j\leq d_\lambda$ and $1\leq k,l
\leq d_\mu$. Then a straightforward computation yields that
\[ \rho^\lambda_{rs}(\bD_{z^{-1}})=\sum_{x,y\in W}  \bc_{\lambda}^{-1}\,
\rho^\lambda_{ls}(\bD_{x^{-1}})\,\rho^\lambda_{rl} (\bD_{y^{-1}})\,
h_{x,y,z}.\]
Further multiplying by $\varepsilon^{\ba(z)}$ and noting that
$\bc_\lambda^{-1}=f_\lambda^{-1}\,\varepsilon^{2\ba_\lambda}/(1+g_\lambda)$
where $g_\lambda \in F[\Gamma_{>0}]$, we obtain
\[ \varepsilon^{\ba(z)}\,\rho^\lambda_{rs}(\bD_{z^{-1}})=
\sum_{x,y\in W} \frac{f_\lambda^{-1}}{1+g_\lambda}
\bigl(\varepsilon^{\ba_\lambda} \rho^\lambda_{ls}(\bD_{x^{-1}})\bigr)\,
\bigl(\varepsilon^{\ba_\lambda} \rho^\lambda_{rl} (\bD_{y^{-1}})\bigr)
\bigl(\varepsilon^{\ba(z)}\,h_{x,y,z}\bigr).\]
Now all terms in the above sum lie in $\cO$, hence the whole sum will lie in
$\cO$ and so $\varepsilon^{\ba(z)}\,\rho^\lambda_{rs}(\bD_{z^{-1}})\in
\cO$.

Now assume, if possible, that $\ba(z)<\ba_\lambda$. Then we could conclude
that the constant term of $\varepsilon^{\ba_\lambda}\,\rho^\lambda_{rs}
(\bD_{z^{-1}})$ is zero, that is, $c_{z^{-1},\lambda}^{rs}=0$, and this 
holds for all $1 \leq r,s\leq d_\lambda$. Since $\rho^\lambda$ is an 
orthogonal representation, \cite[Theorem~4.4]{my02} shows that then we 
also have $c_{z,\lambda}^{rs}=0$ for all $1\leq r,s\leq d_\lambda$, 
a contradiction.
\end{proof}

We will want to find conditions which ensure that we have 
equality in Lemma~\ref{abound}. Consider the following property:
\begin{itemize}
\item[{\bf E1.}] Let $x,y\in W$ and $\lambda,\mu \in \Lambda$ be such that
$E^\lambda \leftrightsquigarrow_{L} x$ and $E^\mu \leftrightsquigarrow_L y$. 
If $x \leq_{L} y$, then $\ba_\mu \leq \ba_\lambda$. 
In particular, if $x \in W$ and $\lambda, \mu \in \Lambda$ are such that  
$E^\lambda \leftrightsquigarrow_{L} x$ and $E^\mu \leftrightsquigarrow_L x$,
then $\ba_\lambda= \ba_\mu$. 
\end{itemize}
Assume that {\bf E1} holds and let $z \in W$. Then we define $\tilde{\ba}
(z)=\ba_\lambda$ where $\lambda\in \Lambda$ is such that $E^\lambda 
\leftrightsquigarrow_{L} z$. Note that $\tilde{\ba}(z)$ is well-defined 
by {\bf E1}.  Furthermore, we have:
\begin{itemize}
\item[{\bf E1'.}] If $x,y\in W$ are such that $x \leq_{\cLR} y$, then 
$\tilde{\ba}(y)\leq \tilde{\ba}(x)$. In particular, $\tilde{\ba}$ is
constant on two-sided cells.
\end{itemize}
Thus, Lemma~\ref{leadcell1} shows that, letting $\fC$ be the left cell 
containing $z\in W$, then 
\[ \tilde{\ba}(z)=\ba_\lambda \quad \mbox{if $E^\lambda$ is a constituent
of $[\fC]_1$}.\]
Now Lusztig \cite[20.6, 20.7]{Lusztig03} shows that, if {\bf P1}--{\bf P15}
hold, then {\bf E1} holds and we have $\ba(z)=\tilde{\ba}(z)$ for all 
$z \in W$. Our aim is to show that {\bf E1} is sufficient to prove the 
equality $\ba(z)=\tilde{\ba}(z)$ for all $z \in W$; see 
Proposition~\ref{cor1} below.  This will be one of the key steps in our 
verification of {\bf P1}--{\bf P15} for $W$ of type $F_4$ and $I_2(m)$.

%
%

\begin{lem} \label{lem1} Assume that {\bf E1} holds.  Let $w \in W$ and 
$\lambda \in \Lambda$.
\begin{itemize}
\item[(a)] If $\rho^\lambda(\bC_w) \neq 0$ then $\tilde{\ba}(w)\leq 
\ba_\lambda$.
\item[(b)] If $\rho^\lambda(\bD_{w^{-1}}) \neq 0$ then $\tilde{\ba}(w)\geq 
\ba_\lambda$.
\item[(c)] We have $\varepsilon^{\tilde{\ba}(w)}\rho^\lambda_{ij}
(\bD_{w^{-1}})\in \cO$ for all $i,j\in \{1,\ldots,d_\lambda\}$.
\end{itemize}
\end{lem}

\begin{proof}
(a) Let $\fC$ be a left cell such that $E^\lambda$ occurs as a 
constituent of $[\fC]_1$. Now, if $\rho^\lambda(\bC_w)\neq 0$, then 
$\bC_w$ cannot act as zero in $[\fC]_A$. Hence, there exist
$x,y\in \fC$ such that $h_{w,x,y}\neq 0$. We have $\tilde{\ba}(x)=
\tilde{\ba}(y)=\ba_\lambda$ by {\bf E1'} and Lemma~\ref{leadcell1}. 
Since, $h_{w,x,y} \neq 0$, we have $y \leq_{\cR} w$ and so $\tilde{\ba}(w)
\leq \tilde{\ba}(y)= \ba_\lambda$ by {\bf E1'}.

(b) Again, let $\fC$ be a left cell such that $E^\lambda$ occurs as a 
constituent of $[\fC]_1$. Now, if $\rho^\lambda(\bD_{w^{-1}})\neq 0$, 
then $\bD_{w^{-1}}$ cannot act as zero in $[\fC]_A$. Hence, there exists 
some $x \in\fC$ such that $\bD_{w^{-1}}\bC_x \neq 0$. We have 
$\tilde{\ba}(x)=\ba_\lambda$ by {\bf E1'} and Lemma~\ref{leadcell1}. Now, 
since $\tau$ is non-degenerate, there exists some $y \in W$ such that 
$\tau(\bD_{w^{-1}} \bC_x \bC_y) \neq 0$. Then we also have $h_{x,y,w}=
\tau(\bC_x \bC_y \bD_{w^{-1}})=\tau(\bD_{w^{-1}}\bC_x \bC_y) \neq 0$ and 
so $w\leq_{\cR} x$. This implies $\ba_\lambda= \tilde{\ba}(x) \leq 
\tilde{\ba}(w)$ by {\bf E1'}.

(c) Since $\rho^\lambda$ is an orthogonal representation, we have
$\varepsilon^{\ba_\lambda}\rho^\lambda_{ij}(\bD_{w^{-1}}) \in \cO$
for all $i,j\in \{1,\ldots,d_\lambda\}$. Hence the assertion follows from
(b).
\end{proof}

\begin{prop} \label{cor1} Assume that {\bf E1} holds. Then $\ba(z)=
\tilde{\ba}(z)$ for all $z \in W$. Furthermore, for $x,y,z\in W$, we have
\[ \gamma_{x,y,z}=\gamma_{y,z,x}=\gamma_{z,x,y}=\sum_{\lambda \in \Lambda}
\sum_{1\leq i,j,k\leq d_\lambda} f_\lambda^{-1} \, c_{x,\lambda}^{ij}\,
c_{y,\lambda}^{jk}\, c_{z,\lambda}^{ki}.\]
\end{prop}

\begin{proof} Assume that $c_{z,\lambda}^{ij} \neq 0$. Now recall that
$c_{z,\lambda}^{ij}$ is the constant term of $\varepsilon^{\ba_\lambda}(T_z)$,
$\varepsilon^{\ba_\lambda}(\bC_z)$ and $\varepsilon^{\ba_\lambda}(\bD_z)$.
Hence, we have $\rho^\lambda(\bC_x) \neq 0$ and $\rho^\lambda(\bD_z)\neq 0$.
So Lemma~\ref{lem1} yields that $\ba(z)=\ba_\lambda=\tilde{\ba}(z)$.

Let $x,y,z\in W$. As in the proof of Lemma~\ref{abound},
we find that 
\begin{align*}
\varepsilon^{\tilde{\ba}(z)} h_{x,y,z^{-1}}&
=\varepsilon^{\tilde{\ba}(z)}\tau(\bC_x\bC_y\bD_z)=
\varepsilon^{\tilde{\ba}(z)}\sum_{\lambda \in \Lambda} \bc_\lambda^{-1} 
\mbox{trace} (\bC_x\bC_y\bD_z,E^\lambda)\\
&=\sum_{\lambda \in \Lambda} \frac{f_\lambda^{-1}}{1+g_\lambda}\, 
\varepsilon^{2\ba_\lambda+ \tilde{\ba}(z)}\,\mbox{trace}\bigl(
\rho^\lambda(\bC_x\bC_y\bD_z) \bigr).
\end{align*}
Now $\rho^\lambda(\bC_x\bC_y\bD_{z})=\rho^\lambda(\bC_x)\rho^\lambda
(\bC_y)\rho^\lambda(\bD_z)$ and so  the above expression equals
\[ \sum_{\lambda \in \Lambda} \sum_{1\leq i,j,k\leq d_\lambda}
\frac{f_\lambda^{-1}}{1+g_\lambda} \bigl(\varepsilon^{\ba_\lambda}
\rho^\lambda_{ij}(\bC_x)\bigr)\, \bigl(\varepsilon^{\ba_\lambda}
\rho^\lambda_{jk}(\bC_y)\bigr)\, \bigl(\varepsilon^{\tilde{\ba}(z)}
\rho^\lambda_{ki} (\bD_z)\bigr).\]
Furthermore, by Lemma~\ref{lem1}(b), we have $\tilde{\ba}(z) \geq 
\ba_\lambda$ for all non-zero terms in the above sum. So the above sum 
can be rewritten as
\[ \sum_{\lambda \in \Lambda\,:\,\ba_\lambda\leq \tilde{\ba}(z)}
\sum_{1\leq i,j,k\leq d_\lambda} \frac{f_\lambda^{-1}\varepsilon^{\tilde{\ba}
(z) -\ba_\lambda}}{1+g_\lambda} \bigl(\varepsilon^{\ba_\lambda}
\rho^\lambda_{ij}(\bC_x) \bigr)\, \bigl(\varepsilon^{\ba_\lambda}
\rho^\lambda_{jk}(\bC_y)\bigr) \, \bigl(\varepsilon^{\ba_\lambda}
\rho^\lambda_{ki} (\bD_z)\bigr).\]
Since each $\rho^\lambda$ is an orthogonal representation, the terms
$\varepsilon^{\ba_\lambda} \rho^\lambda_{ij}(\bC_x)$,
$\varepsilon^{\ba_\lambda} \rho^\lambda_{jk}(\bC_y)$,
$\varepsilon^{\ba_\lambda} \rho^\lambda_{ki} (\bD_z)$ all
lie in $\cO$. Hence, the whole sum lies in $\cO$. First of all, this
shows that $\varepsilon^{\tilde{\ba}(z)}h_{x,y,z^{-1}} \in \cO\cap
\Z[\Gamma]=\Z[\Gamma_{\geq 0}]$ and so $\ba(z)=\ba(z^{-1})\leq 
\tilde{\ba}(z)$ (where the first equality holds by Remark~\ref{MJinvers}). 
The reverse inequality holds by Lemma~\ref{abound}. Thus, we have
shown that $\tilde{\ba}(z)=\ba(z)$. 

Now let us return to the above sum. We have already noted that each
term lies in $\cO$, hence the constant term of the whole sum above can 
be computed term by term. Thus, the contant term of $\varepsilon^{\ba(z)} 
h_{x,y,z^{-1}}$ equals
\[ \sum_{\lambda \in \Lambda\,:\,\ba_\lambda=\tilde{\ba}(z)}
\sum_{1\leq i,j,k\leq d_\lambda} f_\lambda^{-1} \, c_{x,\lambda}^{ij}\,
c_{y,\lambda}^{jk}\, c_{z,\lambda}^{ki}.\]
We note that, in fact, the sum can be extended over all $\lambda \in
\Lambda$. Indeed, if $c_{z,\lambda}^{ki}\neq 0$ for some $\lambda,k,i$,
then $\tilde{\ba}(z)=\ba_\lambda$ by the definition of $\tilde{\ba}(z)$. 
Thus, we have reached the conclusion that 
\[ \gamma_{x,y,z}=\sum_{\lambda \in \Lambda} \sum_{1\leq i,j,k
\leq d_\lambda} f_\lambda^{-1} \, c_{x,\lambda}^{ij}\, c_{y,\lambda}^{jk}\, 
c_{z,\lambda}^{ki}.\]
It remains to notice that the expression on the right hand side is
symmetrical under cyclic permutations of $x,y,z$. This immediately
yields that $\gamma_{x,y,z}=\gamma_{y,z,x}=\gamma_{z,x,y}$.
\end{proof}

\begin{lem} \label{lem4} Assume that {\bf E1} holds. Let $w \in W$. Then 
$\ba(w)\leq \Delta(w)$. Furthermore, 
\[ \sum_{\lambda \in \Lambda} \sum_{1\leq i \leq d_\lambda}
f_\lambda^{-1}\,c_{w,\lambda}^{ii}=\left\{\begin{array}{cl}
n_w & \qquad \mbox{if $w \in \cD$},\\ 0 & \qquad \mbox{otherwise}.
\end{array}\right.\]
\end{lem}

\begin{proof} We use an argument similar to that in the proof of
\cite[Lemma~4.6]{geia06}. First note that $\tau(\bC_w)=
\overline{p}_{1,w}$. So we obtain the identity
\[ \overline{p}_{1,w}=\sum_{\lambda\in \Lambda} \bc_\lambda^{-1}\,
\mbox{trace}(\rho^\lambda(\bC_w))=\sum_{\lambda \in \Lambda}
\sum_{1\leq i\leq d_\lambda} \frac{f_\lambda^{-1}}{1+g_\lambda}
\varepsilon^{\ba_\lambda} \bigl(\varepsilon^{\ba_\lambda}
\rho_{ii}^\lambda(\bC_w)\bigr).\]
By Proposition~\ref{cor1} and Lemma~\ref{lem1}(a), we have $\ba(w)=
\tilde{\ba}(w)\leq \ba_\lambda$ for all non-zero terms in the above
sum. Thus, we obtain
\[ \varepsilon^{-\ba(w)}\overline{p}_{1,w}=
\sum_{\lambda\in\Lambda\,:\,\ba(w)\leq \ba_\lambda}
\sum_{1\leq i\leq d_\lambda} \frac{f_\lambda^{-1}}{1+g_\lambda}
\varepsilon^{\ba_\lambda-\ba(w)} \bigl(\varepsilon^{\ba_\lambda}
\rho_{ii}^\lambda(\bC_w)\bigr).\]
Since each $\rho^\lambda$ is orthogonal, each term
$\varepsilon^{\ba_\lambda} \rho^\lambda_{ii}(\bC_w)$ lies in $\cO$.
This shows, first of all, that $\varepsilon^{-{\ba}(w)}\overline{p}_{1,w}
\in \cO \cap {\Z}[\Gamma]=\Z[\Gamma_{\geq 0}]$ and so ${\ba}(w)\leq
\Delta(w)$, as required. Furthermore, the constant term of the whole sum 
can be determined term by term. Thus, we have
\[ \varepsilon^{-\ba(w)}\overline{p}_{1,w} \equiv \sum_{\lambda\in\Lambda\,:
\,\ba(w)=\ba_\lambda} \sum_{1\leq i\leq d_\lambda} f_\lambda^{-1}
c_{w,\lambda}^{ii}.\]
But then the sum can be extended over all $\lambda \in \Lambda$ because
we have $c_{w,\lambda}^{ii}=0$ unless $\ba(w)=\tilde{\ba}(w)=\ba_\lambda$.
On the other hand, we have $\varepsilon^{-\ba(w)}\overline{p}_{1,w}
\equiv n_w$ if $\ba(w)=\Delta(w)$, and $\varepsilon^{-\ba(w)}
\overline{p}_{1,w} \equiv 0$ if $\ba(w)<\Delta(w)$. 
\end{proof}

\begin{cor} \label{lem4a} Assume that {\bf E1} holds. Then {\bf P1},
{\bf P4}, {\bf P7} and {\bf P8} hold. Furthermore, for any $z \in W$, we 
have $\ba(z)=\ba_\lambda$ where $\lambda \in \Lambda$ is such that 
$E^\lambda\leftrightsquigarrow_{L} z$.
\end{cor}

\begin{proof} By Proposition~\ref{cor1}, we have $\ba(z)=\tilde{\ba}(z)$
and $\gamma_{x,y,z}=\gamma_{y,z,x}$ for all $x,y,z \in W$. Hence, by
{\bf E1'} and Lemma~\ref{lem4}, we have that {\bf P1}, {\bf P4}, {\bf P7}
hold. Finally, note that {\bf P8} is a formal consequence of {\bf P7} 
and Remark~\ref{MJinvers}; see \cite[14.8]{Lusztig03}. 
\end{proof}

\begin{rem} \label{finrem} Assume that {\bf E1} holds. Then
Proposition~\ref{cor1} and Lemma~\ref{lem4} show that $\gamma_{x,y,z}$ 
and $n_w$ ($w \in \cD$) can be recovered from the knowledge of the 
leading matrix coefficients. Consequently, by Remark~\ref{cellJ}, the 
partition of $W$ into left, right and two-sided cells is completely 
determined by the leading matrix coefficients.

This leads to a new approach to contructing Lusztig's asymptotic ring $\bJ$ 
and study its representation theory; see \cite{myedin08} for further details.
\end{rem}

\section{Methods for checking {\bf P1}--{\bf P15}} \label{sec5a}

Our aim now is to formulate a set of conditions which, together with 
{\bf E1} (formulated in the previous section), imply most of the properties 
{\bf P1}--{\bf P15}. Consider the following properties:
\begin{itemize}
\item[{\bf E2.}] Let $x,y\in W$ and $\lambda,\mu \in \Lambda$ be such that
$E^\lambda \leftrightsquigarrow_{L} x$ and $E^\mu \leftrightsquigarrow_L y$.
If $x \leq_{\cLR} y$ and $\ba_\lambda= \ba_\mu$, then $x \sim_{\cLR} y$.
\item[{\bf E3.}] Let $x,y\in W$ be such that $x \leq_{\cL} y$ and
$x \sim_{\cLR} y$, then $x \sim_{\cL} y$.
\item[{\bf E4.}] Let $\fC$ be a left cell of $W$. Then the function
$\fC \rightarrow \Gamma_{\geq 0}$, $w\mapsto \Delta(w)$, reaches its minimum 
at exactly one element of $\fC$.
\end{itemize}
Note that, if {\bf E1} is assumed to hold, then {\bf E2} can be reformulated
as follows:
\begin{itemize}
\item[{\bf E2'.}] If $x,y\in W$ are such that $x \leq_{\cLR} y$ and
$\tilde{\ba}(x)=\tilde{\ba}(y)$, then $x \sim_{\cLR} y$.
\end{itemize}

\begin{rem} \label{remPE1} The relevance of the above set of conditions is
explained as follows.

Assume that, for a given group $W$ and weight function $L \colon W 
\rightarrow \Gamma$, we can compute explicitly all polynomials $p_{y,w}$ 
where $y \leq w$ in $W$ and all polynomials $\mu_{y,w}^s$ where 
$y,w \in W$ and $s \in S$ are such that $sy<y<w<sw$.

Then note that this information alone is sufficient to determine the
pre-order relations $\leq_{\cL}$, $\leq_{\cR}$, $\leq_{\cLR}$ and the
corresponding equivalence relations. Furthermore, we can construct the
representations afforded by the various left cells of $W$. Finally, the 
irreducible representations of $W$ and the invariants $\ba_\lambda$ 
for $\lambda \in \Lambda$ are explicitly known in all cases. Thus, given
the above information alone, we can verify that {\bf E1}--{\bf E4} hold. 
%
\end{rem}

\begin{rem} \label{remPE2} Assume that {\bf P1}--{\bf P15} hold for
$W$. Then {\bf E1}--{\bf E4} hold for $W$.

Indeed, by \cite[20.6, 20.7]{Lusztig03} (whose proofs involve 
{\bf P1}--{\bf P15}), we have $\ba(z)=\ba_\lambda$ if $E^\lambda 
\leftrightsquigarrow_{L} z$ (see also Lemma~\ref{leadcell1}). Hence 
{\bf P4} implies {\bf E1} and {\bf P11} implies {\bf E2}. Furthermore, 
{\bf E3} follows by a combination of {\bf P4} and {\bf P9}. Finally, 
{\bf E4} follows from {\bf P1} and {\bf P13}, where the minimum of the 
$\Delta$-function is reached at the unique element of $\cD$ contained in a
given left cell.
\end{rem}

\begin{lem} \label{lem5} Assume that {\bf P1} holds. Let ${\cD}=\{d \in W 
\mid {\ba}(d)= \Delta(d)\}$. Then
\[ \sum_{d \in {\cD}} {\gamma}_{x^{-1},y,d}\, n_d=\delta_{xy} 
\qquad \mbox{for any $x,y \in W$}.\]
\end{lem}

\begin{proof} As in the proof of \cite[14.5]{Lusztig03}, we compute the 
constant term of $\tau(\bC_{x^{-1}}\bC_{y})$ in two ways. On the one hand, 
we have $\tau(\bC_{x^{-1}}\bC_{y})\in \delta_{xy} +\Z[\Gamma_{>0}]$; hence 
$\tau(\bC_{x^{-1}}\bC_{y})$ has constant term $\delta_{xy}$. On the other
hand, we have
\begin{align*}
\tau(\bC_{x^{-1}}\bC_{y})&=\sum_{z \in W} h_{x^{-1},y,z}\tau(\bC_z)=
\sum_{z \in W} h_{x^{-1},y,z}\, \overline{p}_{1,z}\\
&=\sum_{z \in W} \varepsilon^{\Delta(z)-{\ba}(z)}\,
\bigl(\varepsilon^{{\ba}(z)} h_{x^{-1},y,z}\bigr) \,
\bigl(\varepsilon^{-\Delta(z)}\overline{p}_{1,z} \bigr).
\end{align*}
Now, by the definition of $\Delta(z)$, the term $\varepsilon^{-\Delta(z)}
\overline{p}_{1,z}$ lies in $\Z[\Gamma_{\geq 0}]$ and has constant term
$n_z$.  The term $\varepsilon^{{\ba}(z)} h_{x^{-1},y,z}$ also lies in
$\Z[\Gamma_{\geq 0}]$ and has constant term ${\gamma}_{x^{-1},y, z^{-1}}$.
Finally, by {\bf P1}, we have ${\ba}(z)\leq \Delta(z)$. Hence, the constant 
term of the whole sum can be computed term by term and we obtain
\[ \delta_{xy}=\sum_{z \in W\,:\,{\ba}(z)=\Delta(z)}
{\gamma}_{x^{-1},y,z^{-1}}\, n_z.\]
Now, by \cite[5.6]{Lusztig03}, we have $p_{1,z}=p_{1,z^{-1}}$ and so
$n_z=n_{z^{-1}}$, $\Delta(z)=\Delta(z^{-1})$. Since we also have ${\ba}(z)
=\ba(z^{-1})$ by Remark~\ref{MJinvers}, we can rewrite the above
expression as
\[ \delta_{xy}=\sum_{z \in W\,:\,{\ba}(z)=\Delta(z)}
{\gamma}_{x^{-1},y,z}\,n_z=\sum_{d \in \cD} {\gamma}_{x^{-1},y,d}\,n_d,\]
as desired. 
\end{proof}

\begin{prop} \label{mainprop} Assume that {\bf E1}--{\bf E4} hold for $W$
and all parabolic subgroups of $W$. Then {\bf P1}--{\bf P14} hold for $W$.
\end{prop}

\begin{proof} By Corollary~\ref{lem4a}, we already know that {\bf P1}, 
{\bf P4}, {\bf P7}, {\bf P8} hold. Now let us consider the remaining 
properties.

\smallskip
{\bf P2} Let $x,y \in W$ and assume that $\gamma_{x^{-1},y,d} \neq 0$ for
some $d \in \cD$. First we show that $d$ is uniquely determined by this
condition. Indeed, let $\fC$ be the left cell containing~$x$. By {\bf P8}, 
we have $d \sim_{\cL} x$, i.e., $d \in \fC$. By {\bf P1}, {\bf P4}, 
we have $\Delta(d)=\ba(d)=\ba(w)\leq \Delta(w)$ for all $w \in \fC$. Thus,
the $\Delta$-function, restricted to $\fC$, reaches its minimum at $d$.
Now {\bf E4} shows that $d$ is uniquely determined, as claimed. 

Consequently, the sum in Lemma~\ref{lem5} reduces to one term and we have 
$\gamma_{x^{-1},y,d}n_d=\delta_{xy}$. Since the left hand side is assumed 
to be non-zero, we deduce that $x=y$. 

\smallskip
{\bf P3} Let $y \in W$. By Lemma~\ref{lem5}, there exists some $d \in \cD$ 
such that $\gamma_{y^{-1},y,d}\neq 0$. Arguing as in the proof of {\bf P2},
we see that $d$ is uniquely determined.

\smallskip
{\bf P5} is a formal consequence of {\bf P1}, {\bf P3}; see 
\cite[14.5]{Lusztig03}.

\smallskip
{\bf P6} is a formal consequence of {\bf P2}, {\bf P3}; see 
\cite[14.6]{Lusztig03}.

\smallskip
{\bf P9} Let $x,y\in W$ be such that $x\leq_{\cL} y$ and $\ba(x)=\ba(y)$.
In particular, we have $x\leq_{\cLR} y$ and, by {\bf E1} and 
Proposition~\ref{cor1}, we have $\tilde{\ba}(x)=\tilde{\ba}(y)$. So 
{\bf E2'} implies that $x \sim_{\cLR} y$. Finally, {\bf E3} yields 
$x\sim_{\cL} y$, as required.

\smallskip
{\bf P10} is a formal consequence of {\bf P9}; see \cite[14.10]{Lusztig03}.

\smallskip
{\bf P11} is a formal consequence of {\bf P4}, {\bf P9}, {\bf P10}; 
see \cite[14.11]{Lusztig03}.

\smallskip
{\bf P12} Since {\bf E1}--{\bf E4} are assumed to hold for $W$ and for $W_I$,
we already know that {\bf P1}--{\bf P11} hold for $W$ and $W_I$. Now {\bf P12}
is a formal consequence of {\bf P3}, {\bf P4}, {\bf P8} for $W$ and $W_I$; 
see \cite[14.12]{Lusztig03}.

\smallskip
{\bf P13} Let $\fC$ be a left cell. First we show that $\fC$ contains at most
one element from $\cD$. Let $d \in \fC\cap \cD$. By {\bf P1}, {\bf P4}, we 
have $\Delta(d)=\ba(d)=\ba(w)\leq \Delta(w)$ for all $w \in \fC$. Thus,
the $\Delta$-function (restricted to $\fC$) reaches its minimum at~$d$. So
{\bf E4} shows that $d$ is uniquely determined, as claimed.

Now let $x \in \fC$. By Lemma~\ref{lem5}, there exists some $d \in \cD$ 
such that $\gamma_{x^{-1},x,d}\neq 0$. By {\bf P8}, we have $d \in 
\fC$ and so $d \in \fC\cap \cD$. By the previus argument, $\fC\cap \cD
=\{d\}$. 

\smallskip
{\bf P14} is a formal consequence of {\bf P6}, {\bf P13}; 
see \cite[14.14]{Lusztig03}.
\end{proof}

Finally, we discuss the remaining property in Lusztig's list which 
is not covered by the above arguments: property {\bf P15}. 

\begin{rem} \label{p15equal} Assume that we are in the equal parameter
case. Then, by \cite[14.15 and 15.7]{Lusztig03}, {\bf P15} can be deduced
once {\bf P4}, {\bf P9} and {\bf P10} are known to hold. Hence, in this
case, all of {\bf P1}--{\bf P15} are a consequence of {\bf E1}--{\bf E4}.
\end{rem}

The following two results will be useful in dealing with {\bf P15} in
the case of unequal parameters.

\begin{rem} \label{proofp15} Following \cite[14.15]{Lusztig03}, we
can reformulate {\bf P15} as follows. Let $\breve{\Gamma}$ be an isomorphic 
copy of $\Gamma$; then $L$ induces a weight function $\breve{L}\colon 
W \rightarrow \breve{\Gamma}$. Let $\breve{\bH}=\bH_{\breve{A}}(W,S,
\breve{L})$ be the corresponding Iwahori--Hecke algebra over 
$\breve{A}=R[\breve{\Gamma}]$, with parameters $\{\breve{v}_s \mid 
s \in S\}$. We have a corresponding Kazhdan--Lusztig basis
$\{\breve{\bC}_w \mid w \in W\}$. We shall regard $A$ and $\breve{A}$
as subrings of $\cA=R[\Gamma\oplus \breve{\Gamma}]$. By extension
of scalars, we obtain $\cA$-algebras $\bH_{\cA}=\cA \otimes_A \bH$
and $\breve{\bH}_{\cA}=\cA \otimes_{\breve{A}} \breve{\bH}$. Let
$\cE$ be the free $\cA$-module with basis $\{e_w \mid w \in W\}$.
We have an obvious left $\bH_{\cA}$-module structure and an obvious
right $\breve{\bH}_{\cA}$-module structure on $\cE$ (induced by
left and right multiplication). Now consider the following condition, 
where $s,t \in S$ and $w \in W$:
\begin{equation*}
(\bC_s.e_w).\breve{\bC}_t-\bC_s.( e_w.\breve{\bC}_t)= \mbox{combination of 
$e_y$ where $y \leq_{\cLR} w$, $y \not\sim_{\cLR} w$}.\tag{$*$}
\end{equation*}
As remarked in \cite[14.15]{Lusztig03}, ($*$) is already known to hold
if $sw<w$ or $wt<w$. Hence, it is sufficient to consider ($*$) for the 
cases where both $sw>w$ and $wt>w$.

The discussion in \cite[14.15]{Lusztig03} shows that {\bf P15} is equivalent
to ($*$), provided that {\bf P4}, {\bf P11} are already known to hold. 
\end{rem}

By looking at the proof of Theorems~\ref{MJasymp}, one notices that it 
only requires a property which looks weaker than {\bf P15}; we called this 
property {\bf P15'} in \cite[\S 5]{mylaus}. The following result shows 
that, in fact, {\bf P15} is equivalent to {\bf P15'}.

\begin{lem} \label{Mp15prime} Assume that {\rm {\bf P1}, {\bf P4}, {\bf P7},
{\bf P8}} hold.  Then {\bf P15} is equivalent to the following property
\begin{itemize}
\item[\bf P15$^\prime$.] If $x,x',y,w\in W$ satisfy $\ba(w)=\ba(y)$, then
\[\sum_{u \in W} \gamma_{w,x',u^{-1}}\, h_{x,u,y} =\sum_{u\in W}
h_{x,w,u}\, \gamma_{u,x',y^{-1}}.\]
\end{itemize}
Note that, on both sides, the sum needs only be extended over all $u \in W$
such that $\ba(u)=\ba(w)=\ba(y)$ (thanks to {\bf P4}).
\end{lem}

\begin{proof} First note that {\bf P15$^\prime$} appears in
\cite[18.9(b)]{Lusztig03}, where it is deduced from {\bf P4}, {\bf P15}. Now
we have to show that, conversely, {\bf P1}, {\bf P4}, {\bf P7}, {\bf P8} and 
{\bf P15$^\prime$} imply {\bf P15}. First we claim that {\bf P15$^\prime$} 
implies the following statement (which appears in \cite[18.10]{Lusztig03}):

\smallskip
{\em If $x,y,y'\in W$ are such that $\ba(y)=\ba(y')$, then}
\begin{equation*}
h_{x,y',y}=\sum_{\atop{d \in \cD,z \in W}{\ba(d)=\ba(z)}}
h_{x,d,z}\,n_d\, \gamma_{z,y',y^{-1}}. \tag{$*$}
\end{equation*}
To see this, note that on the right hand side, we may replace
the condition $\ba(d)=\ba(z)$ by the condition $\ba(d)=\ba(y')$; see
{\bf P4}, {\bf P8}. Using also {\bf P15$^\prime$} (where $w=d\in \cD$ 
and $x'$ is replaced by $y'$), we see that the right hand side of
($*$) equals
\[ \sum_{d\in \cD\,:\, \ba(d)=\ba(y')} n_d \Bigl(\sum_{z \in W} h_{x,d,z}\,
\gamma_{z,y',y^{-1}}\Bigr)= \sum_{d\in \cD\,:\, \ba(d)=\ba(y')} n_d 
\Bigl(\sum_{z \in W} \gamma_{d,y',z^{-1}}\, h_{x,z,y}\Bigr).\]
Now $\gamma_{d,y',z^{-1}}=0$ unless $\ba(d)=\ba(y')$; see {\bf P8}, 
{\bf P4}. Using also {\bf P7} and Lemma~\ref{lem5}, the right hand side of 
the above equation can be rewritten as
\[ \sum_{z \in W} h_{x,z,y} \Bigl(\sum_{d \in \cD} \gamma_{y',z^{-1},d}\,
n_d\Bigr)=\sum_{z\in W} h_{x,z,y}\, \delta_{zy'}=h_{x,y',y}.\]
Thus, ($*$) is proved.

Now consider the left hand side in {\bf P15} where 
$x,x,'y,w \in W$ are such that $a:=\ba(w)= \ba(y)$. If $h_{w,x',y'} \neq 0$ 
then $y' \leq_{\cR} w$ and so $a=\ba(w)\leq \ba(y')$ by {\bf P4}; similarly, 
if $h_{x,y',y} \neq 0$, then $y \leq_{\cL} y'$ and so $\ba(y') \leq
\ba(y)=a$. Hence, $\ba(y')=a$, and so we may assume that the sum only runs
over all $y'\in W$ such that $\ba(y')=a$. Inserting now ($*$) into
the left hand side of {\bf P15}, we obtain the expression
\[\sum_{\atop{y'\in W}{\ba(y')=a}}\sum_{\atop{d\in\cD,z\in W}{\ba(d)=
\ba(z)}} \gamma_{z,y',y^{-1}}\,h_{w,x',y'}\otimes h_{x,d,z}\, n_d
=\sum_{\atop{d \in \cD,z \in W}{\ba(d)=\ba(z)}} \Bigl(
\sum_{\atop{y'\in W}{\ba(y')=a}} \gamma_{z,y',y^{-1}}\, h_{w,x',y'}
\Bigr) \otimes h_{x,d,z}\, n_d.\]
Now, using Remark~\ref{MJinvers}
and {\bf P15$^\prime$}, we can rewrite the interior sum as follows:
\begin{multline*}
\sum_{\atop{y'\in W}{\ba(y')=a}} \gamma_{z,y',y^{-1}}\, h_{w,x',y'}=
\sum_{\atop{y'\in W}{\ba(y')=a}} \gamma_{y'^{-1},z^{-1},y}\,
h_{x'^{-1},w^{-1},y'^{-1}}\\
=\sum_{\atop{u\in W}{\ba(u)=a}} h_{x'^{-1},w^{-1},u}\, \gamma_{u,z^{-1},y}
=\sum_{\atop{u\in W}{\ba(u)=a}} \gamma_{w^{-1},z^{-1},u^{-1}}\,
h_{x'^{-1},u,y^{-1}}\\ =\sum_{\atop{u \in W}{\ba(u)=a}} \gamma_{z,w,u}\,
h_{u^{-1},x',y} =\sum_{\atop{u\in W}{\ba(u)=a}}\gamma_{z,w,u^{-1}}\,
h_{u,x',y}.
\end{multline*}
Inserting this back into the above expression, we find that
\[ \sum_{\atop{d \in \cD,z \in W}{\ba(d)=\ba(z)}} 
\Bigl(\sum_{\atop{y'\in W}{\ba(y')=a}} \gamma_{z,y',y^{-1}}\, h_{w,x',y'}
\Bigr) \otimes h_{x,d,z}\,n_d=\sum_{\atop{d \in \cD,z \in W}{\ba(d)=
\ba(z)}} \Bigl(\sum_{\atop{u\in W}{\ba(u)=a}} \gamma_{z,w,u^{-1}}\, 
h_{u,x',y}\Bigr) \otimes h_{x,d,z}\,n_d.  \]
Using also ($*$), we obtain the expression
\[ \sum_{\atop{u\in W}{\ba(u)=a}} h_{u,x',y}\otimes \Bigl(\sum_{\atop{d 
\in \cD,z \in W}{\ba(d)=\ba(z)}}\gamma_{z,w,u^{-1}}\, h_{x,d,z}\,n_d\Bigr)
= \sum_{\atop{u \in W}{\ba(u)=a}} h_{u,x',y}\otimes h_{x,w,u},\]
which is the right hand side of {\bf P15}. Note that, in the right hand
side of {\bf P15}, the sum need only be extended over all $y'\in W$ such
that $\ba(y')=a$. (The argument is similar to the one we used to prove the
analogous statement for the left hand side.)
\end{proof}

\begin{exmp} \label{h34} Assume that $(W,S)$ is of type $H_4$. Then we are 
in the equal parameter case. So, in order to verify {\bf P1}--{\bf P15}, 
it is sufficient to verify {\bf E1}--{\bf E4}; see Remark~\ref{p15equal}. 
Now Alvis \cite{Alvis87} has computed all polynomials $p_{y,w}$ where 
$y\leq w$ in $W$. Since we are in the equal parameter case, this also 
determines all polynomials $\mu_{y,w}^s$ where $y,w \in W$ and $s \in S$ 
are such that $sy<y<w<sw$; see \cite[6.5]{Lusztig03}. In this way, Alvis
explicitly determined the relations $\leq_{\cL}$ and $\leq_{\cLR}$; he also
found the decomposition of the left cell representations into irreducibles.

It turns out that the partial order induced on the set of two-sided cells
is a total order. (I thank Alvis for having verified this using the data 
in \cite{Alvis87}.) With the notation in [{\em loc.\ cit.}], this total
order is given by:
\begin{align*}
 G^* \leq_{\cLR} F^* \leq_{\cLR} E^* &\leq_{\cLR} D^* \leq_{\cLR} C^* 
\leq_{\cLR} B^* \leq_{\cLR} A^*\\ &=A \leq_{\cLR} B \leq_{\cLR} C 
\leq_{\cLR} D \leq_{\cLR}E\leq_{\cLR} F \leq_{\cLR} G.\end{align*}
Comparing with the information on the invariants  $\ba_\lambda$ provided
by Alvis--Lusztig \cite{AlLu82}, we see that {\bf E1} and {\bf E2} hold.
Furthermore, {\bf E3} is already explicitly stated in 
\cite[Cor.~3.3]{Alvis87}.  Finally, {\bf E4} is readily checked using
Alvis' computation of the left cells and the polynomials $p_{y,w}$.

In this way, we obtain an alternative proof of {\bf P1}--{\bf P15} for
$H_4$, which does not rely on DuCloux's computation \cite{Fokko} of all 
structure constants $h_{x,y,z}$ ($x,y,z\in W$). 

Similar arguments can of course also be applied to $(W,S)$ of type $H_3$.
\end{exmp}

\section{Lusztig's homomorphism} \label{sec6}

We now use the methods developped in the previous section to verify
{\bf P1}--{\bf P15} for type $F_4$ and $I_2(m)$. Then we are in a position
to extend the construction of Lusztig's isomorphism to the general 
case of unequal parameters.

\begin{prop} \label{p115i2} Let $3 \leq m<\infty$ and $(W,S)$ be of type 
$I_2(m)$, with generators $s_1,s_2$ such that $(s_1s_2)^m=1$. Then 
{\bf P1}--{\bf P15} hold for any weight function $L\colon W \rightarrow
\Gamma$ and any monomial order $\leq$ such that $L(s_i)>0$ for $i=1,2$.
\end{prop}

\begin{proof} If $L(s_1)=L(s_2)$, this is proved by DuCloux \cite{Fokko}, 
following the approach in \cite[17.5]{Lusztig03} (concerning the infinite 
dihedral group). Now assume that $L(s_1)\neq L(s_2)$; in particular, 
$m\geq 4$ is even. Without loss of generality, we can assume that $L(s_1)>
L(s_2)$. It is probably possible to use arguments similar to those in 
\cite{Fokko} and \cite[17.5]{Lusztig03} (which essentially amount to
computing all structure constants $h_{x,y,z}$). However, in the present
case, it is rather straightforward to verify {\bf E1}--{\bf E4}. Indeed, 
by \cite[\S 5.4]{gepf}, we have 
\[ \Irr(W)=\{1_W,\varepsilon,\varepsilon_1,\varepsilon_2,\rho_1,\rho_2,
\ldots, \rho_{(m-2)/2)}\},\]
where $1_W$ is the trivial representation, $\varepsilon$ is the 
sign representation, $\varepsilon_1,\varepsilon_2$ are two further 
$1$-dimensional representations, and all $\rho_j$ are $2$-dimensional.
We fix the notation such that $s_1$ acts as $+1$ in $\varepsilon_1$
and as $-1$ in $\varepsilon_2$. Using \cite[8.3.4]{gepf}, we find
\begin{alignat*}{2}
\ba_{1_W} &= 0,&\qquad  f_{1_W} &= 1,\\
\ba_{\varepsilon_1} &= L(s_2),&\quad  f_{\varepsilon_1} &= 1,\\
\ba_{\rho_j} &= L(s_1), & \quad f_{\rho_j}&=\frac{m}{2-\zeta^{2j}-
\zeta^{-2j}} \qquad \mbox{for all $j$},\\
\ba_{\varepsilon_2} &= \frac{m}{2}\bigl(L(s_1)-L(s_2)\bigr)+L(s_2),
&\quad  f_{\varepsilon_2}&=1,\\
\ba_{\varepsilon} &= \frac{m}{2}\bigl(L(s_1)+L(s_2)\bigr)&\quad  
f_{\varepsilon}&=1;
\end{alignat*}
where $\zeta\in \C$ is a root of unity of order $m$.
Observe that, in the above list, the $\ba$-values are in strictly 
increasing order from top to bottom.

Now, by \cite[6.6, 7.5, 7.6]{Lusztig03} and \cite[Exc.~11.3]{gepf}, 
we have the following multiplication rules for the Kazhdan--Lusztig 
basis. For any $k \geq 0$, write $1_k=s_1s_2s_1 \cdots $ ($k$ factors) 
and $2_k=s_2s_1s_2 \cdots$ ($k$ factors). Given $k,l\in \Z$, we define 
$\delta_{k>l}$ to be $1$ if $k>l$ and to be $0$ otherwise. Then 
\begin{align*}
\bC_{1_1}\bC_{1_{k+1}}& =(v_{s_1}+v_{s_1}^{-1})\bC_{1_{k+1}},\\
\bC_{2_1}\bC_{2_{k+1}}&=(v_{s_2}+v_{s_2}^{-1})\bC_{2_{k+1}},\\
\bC_{2_1}\bC_{1_k}&=\bC_{2_{k+1}},\\
\bC_{1_1}\bC_{2_k}&=\bC_{1_{k+1}}+\delta_{k>1}\zeta\bC_{1_{k-1}}+
\delta_{k>3}\bC_{1_{k-3}},
\end{align*}
for any $0 \leq k <m$, where $\zeta=v_{s_1}v_{s_2}^{-1}+ 
v_{s_1}^{-1}v_{s_2}$.  Using this information, the pre-order relations 
$\leq_{\cL}$, $\leq_{\cR}$ and $\leq_{\cLR}$ are easily and explicitly 
determined; see \cite[8.8]{Lusztig03}. The two-sided cells and the partial 
order on them are  given by
\begin{equation*}
\{1_m\}\; \leq_{\cLR} \; \{1_{m-1}\} \; \leq_{\cLR} \;
W\setminus \{1_0,2_1,1_{m-1},1_m\} \; \leq_{\cLR} \; \{2_1\} \;
\leq_{\cLR}  \; \{1_{0}\}.  \tag{$\heartsuit$}
\end{equation*}
The set $W\setminus \{1_0,2_1,1_{m-1},1_m\}$ consists of two left cells,
$\{1_1,2_2,1_3,\ldots,2_{m-2}\}$ and $\{1_2,2_3,1_4, \ldots,2_{m-1}\}$, but 
these are not related by $\leq_{\cL}$. (If they were, then, by 
\cite[8.6]{Lusztig03}, the right descent set of the elements in one of 
them would have to be contained in the right descent set of the elements 
in the other one---which is not the case.) The other two-sided cells are 
just left cells. In particular, we see that {\bf E3} holds. 

Now we can also construct the representations given by the various 
left cells and decompose them into irreducibles; we obtain:
\begin{align*}
\{1_0\} &\quad \mbox{affords} \quad 1_W,\\
\{2_1\} &\quad \mbox{affords} \quad \varepsilon_1,\\
\{1_1,2_2,1_3,\ldots,2_{m-2}\} &\quad \mbox{affords} \quad 
\rho_1+\rho_2+\cdots +\rho_{(m-2)/2},\\
\{1_2,2_3,1_4, \ldots,2_{m-1}\}&\quad \mbox{affords} \quad 
\rho_1+\rho_2+\cdots +\rho_{(m-2)/2},\\
\{1_{m-1}\} &\quad \mbox{affords} \quad \varepsilon_2,\\
\{1_m\} &\quad \mbox{affords} \quad \varepsilon.
\end{align*}
Using this list and the above information on the $\ba$-values and
the partial order on the two-sided cells, we see that {\bf E1} and
{\bf E2} hold. 

Next, by \cite[7.4, 7.6]{Lusztig03} and \cite[Exc.~11.3]{gepf}, the
polynomials $p_{y,w}$ are explicitly known. Thus, we can determine the 
function $w \mapsto \Delta(w)$. We obtain 
\begin{align*}
\Delta(1_{2k}) = \Delta(2_{2k})&=kL(s_1)+kL(s_2) \quad \mbox{if 
$k \geq 0$},\\ 
\Delta(2_{1}) &= L(s_2)\\
\Delta(1_{2k+1}) &= (k+1)L(s_1) -kL(s_2)\quad \mbox{if $k\geq 0$}\\
\Delta(2_{2k+1}) &= kL(s_1)+(k-1)L(s_2) \quad \mbox{if $k\geq 1$}.
\end{align*}
Thus, we see that {\bf E4} holds. In the left cell $\{1_1,2_2,1_3,\ldots,
2_{m-2}\}$, the function $\Delta$ reaches its minimum at $1_1$; 
in the left cell $\{1_2,2_3,1_4, \ldots,2_{m-1}\}$, the minimum is reached
at $2_3$. We see that 
\begin{gather*}
\cD=\{1_0,\;2_1,\;1_1,\;2_3,\;1_{m-1},\;1_m\},\\
n_{1_0}=n_{2_1}=n_{1_1}=n_{2_3}=n_{1_m}=+1, \; n_{1_{m-1}}=-1.
\end{gather*}
Thus, we have verified that {\bf E1}--{\bf E4} hold for $W$. We also know
that {\bf P1}--{\bf P15} hold for every proper parabolic subgroup of
$W$. (Note that the only proper parabolic subgroups of $W$ are $\langle s_1
\rangle$ and $\langle s_2\rangle$.)  Hence, by Remark~\ref{remPE2} and 
Proposition~\ref{mainprop}, we can conclude that {\bf P1}--{\bf P14} hold 
for $W$.

It remains to verify {\bf P15}. For this purpose, we must check that 
condition ($*$) in Remark~\ref{proofp15} holds for all $w \in W$ and $i,j 
\in \{1,2\}$ such that $s_iw>w$, $ws_j>w$. A similar verification is 
done by Lusztig \cite[17.5]{Lusztig03} for the infinite dihedral group. 
We notice that the same arguments also work in our situation if $w$ is 
such that we do not encounter the longest element $w_0=1_m=2_m$ in the 
course of the verification. This certainly is the case if $l(w)<m-2$. 
Thus, we already know that ($*$) holds when $l(w)<m-2$. It remains to
verify ($*$) when $l(w)$ equals $m-2$ or $m-1$, that is, when $w\in 
\{1_{m-2},2_{m-2}, 1_{m-1},2_{m-1}\}$. 

Assume first that $w=1_{m-2}$. The left descent set of $w$ is $\{s_1\}$ 
and, since $m$ is even, the right descent set of $w$ is $\{s_2\}$. So we 
must check ($*$) with $s=s_2$ and $t=s_1$. Using the above multiplication 
formulas, we find:
\[(\bC_{2_1}.e_{1_{m-2}}).\breve{\bC}_{1_1}=e_{2_{m-1}}.\breve{\bC}_{1_1}.\]
Now, since $m$ is even, $\{s_2\}$ is the right descent set of $1_{m-1}$.
Hence right-handed versions of the above multiplication rules imply that
\[ (\bC_{2_1}.e_{1_{m-2}}).\breve{\bC}_{1_1}=e_{2_{m-1}}.\breve{\bC}_{1_1}
=e_{2_m}+\delta_{m>2}\breve{\zeta} e_{2_{m-2}}+\delta_{m>4}e_{2_{m-4}},\]
where $\breve{\zeta}=\breve{v}_{s_1}\breve{v}_{s_2}^{-1}+
\breve{v}_{s_1}^{-1}\breve{v}_{s_2}$. On the other hand, we have
\begin{align*}
\bC_{2_1}.(e_{1_{m-2}}.\breve{\bC}_{1_1})&=
\bC_{2_1}.\bigl(e_{1_{m-1}}+\delta_{m>3} \breve{\zeta} e_{1_{m-3}} +
\delta_{m>5}e_{2_{m-5}}\bigr)\\
&= e_{2_m}+\delta_{m>3} \breve{\zeta} e_{2_{m-2}}+\delta_{m>5}e_{2_{m-4}}.
\end{align*}
Now note that, since $m$ is even, we have $\delta_{m>3}=\delta_{m>2}$ and
$\delta_{m>4}=\delta_{m>5}$. Hence, we actually see that the expression
in ($*$) is zero.

Now assume that $w=2_{m-2}$. Then we must check ($*$) with $s=s_1$ and
$t=s_2$. Arguing as above, we find that
\begin{align*}
(\bC_{1_1}.e_{2_{m-2}}).\breve{\bC}_{2_1}&=\bigl(e_{1_{m-1}}+
\delta_{m>3}\zeta e_{1_{m-3}}+\delta_{m>5}e_{m-5}\bigr).\breve{\bC}_{2_1}
\\&=e_{1_{m}}+ \delta_{m>3}\zeta e_{1_{m-2}}+\delta_{m>5}e_{m-4},\\
\bC_{1_1}.(e_{2_{m-2}}.\breve{\bC}_{2_1})&=\bC_{1_1}.e_{2_{m-1}}
=e_{1_{m}}+ \delta_{m>2}\zeta e_{1_{m-2}}+\delta_{m>4}e_{m-4}.
\end{align*}
Again, we see that the difference of these two expressions is zero.

Next, let $w=1_{m-1}$. Then we must check ($*$) with $s=t=s_2$. We
obtain
\begin{align*}
(\bC_{2_1}.e_{1_{m-1}}).\breve{\bC}_{2_1}&=e_{2_m}.\breve{\bC}_{2_1}=
(\breve{v}_{s_2}+\breve{v}_{s_2}^{-1})e_{2_m},\\
\bC_{2_1}.(e_{1_{m-1}}.\breve{\bC}_{2_1})&=\bC_{2_1}.e_{1_m}=
(v_{s_2}+v_{s_2}^{-1})e_{2_m}.
\end{align*}
Hence the difference of these two expressions is a scalar multiple
of $e_{2_m}$. The description of $\leq_{\cLR}$ in ($\heartsuit$) now 
shows that ($*$) holds.

Finally, let $w=2_{m-1}$. Then we must check ($*$) with $s=t=s_1$.  We find
\[ (\bC_{1_1}.e_{2_{m-1}}).\breve{\bC}_{1_1}=\bigl(e_{1_{m}}+\delta_{m>2}
\zeta e_{1_{m-2}} +\delta_{m>4}e_{1_{m-4}}\bigr).\breve{\bC}_{1_1}.\]
Furthermore, we obtain:
\begin{align*}
e_{1_m}.\breve{\bC}_{1_1}&=(\breve{v}_{s_1}+\breve{v}_{s_1}^{-1})e_{1_m},\\
e_{1_{m-2}}.\breve{\bC}_{1_1}&=e_{1_{m-1}}+ \delta_{m>3}
\breve{\zeta} e_{1_{m-3}} + \delta_{m>5}e_{1_{m-5}}, \\ 
e_{1_{m-4}}.\breve{\bC}_{1_1}&=e_{1_{m-3}}+ \delta_{m>5}\breve{\zeta} 
e_{1_{m-5}} +\delta_{m>7}e_{1_{m-7}}.
\end{align*}
Inserting this into the above expression, we obtain
\begin{multline*}
(\bC_{1_1}.e_{2_{m-1}}).\breve{\bC}_{1_1}= (\breve{v}_{s_1}+ 
\breve{v}_{s_1}^{-1})e_{1_m}+ \delta_{m>2}\zeta e_{1_{m-1}}\\
+(\delta_{m>3}\zeta \breve{\zeta}+\delta_{m>4}) e_{1_{m-3}} +  
(\zeta+\breve{\zeta})\delta_{m>5}e_{1_{m-5}}+\delta_{m>7} e_{1_{m-7}}.
\qquad\qquad \end{multline*}
A similar computation yields
\begin{multline*}
\bC_{1_1}.(e_{2_{m-1}}.\breve{\bC}_{1_1})= (v_{s_1}+v_{s_1}^{-1})e_{1_m}+ 
\delta_{m>2}\breve{\zeta} e_{1_{m-1}}\\
+(\delta_{m>3}\zeta \breve{\zeta}+\delta_{m>4}) e_{1_{m-3}} +
(\zeta+\breve{\zeta})\delta_{m>5}e_{1_{m-5}}+\delta_{m>7} e_{1_{m-7}}
\qquad \qquad 
\end{multline*}
and so 
\[(\bC_{1_1}.e_{2_{m-1}}).\breve{\bC}_{1_1}-
\bC_{1_1}.(e_{2_{m-1}}.\breve{\bC}_{1_1})=(\breve{v}_{s_1}+
\breve{v}_{s_1}^{-1}-v_{s_1}-v_{s_1}^{-1})e_{1_m}+ \delta_{m>2}(\zeta -
\breve{\zeta})e_{1_{m-1}}.\]
The description of $\leq_{\cLR}$ in ($\heartsuit$) now shows that 
($*$) holds.

Thus, we have verified that {\bf P15} holds.
\end{proof}

\begin{table}[htbp] 
\caption{The invariants $f_\lambda$ and $\ba_\lambda$
for type $F_4$} \label{invF4}
{\footnotesize \begin{center}
$\begin{array}{|c|cc|cr|cc|cc|} \hline & \multicolumn{2}{c|}{b{>}2a{>}0}
& \multicolumn{2}{c|}{b{=}2a{>}0} & \multicolumn{2}{c|}{2a{>}b{>}a{>}0} &
\multicolumn{2}{c|}{b{=}a{>}0} \\
E^\lambda& f_\lambda & \ba_\lambda & f_\lambda & \ba_\lambda & f_\lambda
& \ba_\lambda & f_\lambda & \ba_\lambda \\ \hline
1_1 &1&0        &1&   0 & 1 & 0        & 1 & 0   \\
1_2 &1&12b{-}9a &2& 15a & 1 & 11b{-}7a & 8 & 4a   \\
1_3 &1&3a       &2&  3a & 1 & {-}b{+}5a& 8 & 4a \\
1_4 &1&12b{+}12a&1& 36a & 1 & 12b{+}12a& 1 &24a   \\
2_1 &1&3b{-}3a  &2&  3a & 1 & 2b{-}a   & 2 & a   \\
2_2 &1&3b{+}9a  &2& 15a & 1 & 2b{+}11a & 2 &13a  \\
2_3 &1&a        &1&   a & 1 & a        & 2 & a  \\
2_4 &1&12b{+}a  &1& 25a & 1 & 12b{+}a  & 2 &13a   \\
4_1 &2&3b{+}a   &2&  7a & 2 & 3b{+}a   & 8 & 4a  \\
9_1 &1&2b{-}a   &2&  3a & 1 & b{+}a    & 1 & 2a \\
9_2 &1&6b{-}2a  &1& 10a & 1 & 6b{-}2a  & 8 & 4a  \\
9_3 &1&2b{+}2a  &1&  6a & 1 & 2b{+}2a  & 8 & 4a  \\
9_4 &1&6b{+}3a  &2& 15a & 1 & 5b{+}5a  & 1 &10a  \\
6_1 &3&3b{+}a   &3&  7a & 3 & 3b{+}a   & 3 & 4a  \\
6_2 &3&3b{+}a   &3&  7a & 3 & 3b{+}a   &12 & 4a  \\
12_1&6&3b{+}a   &6&  7a & 6 & 3b{+}a   &24 & 4a  \\
4_2 &1&b        &1&  2a & 1 & b        & 2 &  a \\
4_3 &1&7b{-}3a  &1& 11a & 1 & 7b{-}3a  & 4 & 4a  \\
4_4 &1&b{+}3a   &1&  5a & 1 & b{+}3a   & 4 & 4a \\
4_5 &1&7b{+}6a  &1& 20a & 1 & 7b{+}6a  & 2 &13a  \\
8_1 &1&3b       &1&  6a & 1 & 3b       & 1 & 3a  \\
8_2 &1&3b{+}6a  &1& 12a & 1 & 3b{+}6a  & 1 & 9a  \\
8_3 &1&b{+}a    &2&  3a & 1 & 3a       & 1 & 3a \\
8_4 &1&7b{+}a   &2& 15a & 1 & 6b{+}3a  & 1 & 9a  \\
16_1&2&3b{+}a   &2&  7a & 2 & 3b{+}a   & 4 & 4a \\ \hline
\multicolumn{9}{c}{\mbox{(This table corrects some errors concerning 
$f_\lambda$ in \cite[Table~1]{mylaus}.)}}
\end{array}$
\end{center}} 
\end{table}

\begin{prop} \label{p115f4} Let $(W,S)$ be of type $F_4$ with generators 
and diagram given by:
\begin{center}
\begin{picture}(200,20)
\put( 10, 5){$F_4$}
\put( 61,13){$s_1$}
\put( 91,13){$s_2$}
\put(121,13){$s_3$}
\put(151,13){$s_4$}
\put( 65, 5){\circle*{5}}
\put( 95, 5){\circle*{5}}
\put(125, 5){\circle*{5}}
\put(155, 5){\circle*{5}}
\put(105,2.5){$>$}
\put( 65, 5){\line(1,0){30}}
\put( 95, 7){\line(1,0){30}}
\put( 95, 3){\line(1,0){30}}
\put(125, 5){\line(1,0){30}}
\end{picture}
\end{center}
Then {\bf P1}--{\bf P15} hold for any weight function $L \colon W \rightarrow
\Gamma$ and any monomial order $\leq$ such that $L(s_i)>0$ for $i=1,2,3,4$.
\end{prop}

\begin{table}[htbp]
\caption{Partial order on two-sided cells in type $F_4$} 
\label{lrgraph} 
{\footnotesize \begin{center}
\renewcommand{\arraystretch}{1.3} \begin{tabular}{l}
\begin{picture}(325,265)
\put( 16,  0){$a=b$}
\put( 27, 20){\circle{4}}
\put( 32, 18){\scriptsize{$1_4$}}
\put( 27, 22){\line(0,1){16}}
\put( 27, 40){\circle{4}}
\put( 32, 38){\scriptsize{\fbox{$4_5$}}}
\put( 27, 42){\line(0,1){16}}
\put( 27, 60){\circle{4}}
\put( 32, 58){\scriptsize{$9_4$}}
\put( 28, 62){\line(4,5){13}}
\put( 26, 62){\line(-4,5){13}}
\put( 12, 80){\circle{4}}
\put(  0, 78){\scriptsize{$8_2$}}
\put( 42, 80){\circle{4}}
\put( 47, 78){\scriptsize{$8_4$}}
\put( 13, 82){\line(4,5){13}}
\put( 41, 82){\line(-4,5){13}}
\put( 27,100){\circle{4}}
\put(  0, 99){\scriptsize{\fbox{$12_1$}}}
\put( 28,102){\line(4,5){13}}
\put( 26,102){\line(-4,5){13}}
\put( 12,120){\circle{4}}
\put(  0,118){\scriptsize{$8_1$}}
\put( 42,120){\circle{4}}
\put( 47,118){\scriptsize{$8_3$}}
\put( 13,122){\line(4,5){13}}
\put( 41,122){\line(-4,5){13}}
\put( 27,140){\circle{4}}
\put( 32,140){\scriptsize{$9_1$}}
\put( 27,142){\line(0,1){16}}
\put( 27,160){\circle{4}}
\put( 32,158){\scriptsize{\fbox{$4_2$}}}
\put( 27,162){\line(0,1){16}}
\put( 27,180){\circle{4}}
\put( 32,178){\scriptsize{$1_1$}}
\put( 93,  0){$b=2a$}
\put(105, 20){\circle{4}}
\put(110, 18){\scriptsize{$1_4$}}
\put(105, 22){\line(0,1){16}}
\put(105, 40){\circle{4}}
\put(110, 38){\scriptsize{$2_4$}}
\put(105, 42){\line(0,1){16}}
\put(105, 60){\circle{4}}
\put(110, 58){\scriptsize{$4_5$}}
\put(105, 62){\line(0,1){16}}
\put(105, 80){\circle{4}}
\put( 84, 78){\scriptsize{\fbox{$1_2$}}}
\put(106, 82){\line(1,1){14}}
\put(104, 82){\line(-1,2){11.7}}
\put(121, 98){\circle{4}}
\put(126, 96){\scriptsize{$4_3$}}
\put(121,100){\line(0,1){14}}
\put(121,116){\circle{4}}
\put(126,114){\scriptsize{$9_2$}}
\put( 91,107){\circle{4}}
\put( 79,105){\scriptsize{$8_2$}}
\put(106,132){\line(1,-1){14}}
\put(104,132){\line(-1,-2){11.7}}
\put(105,134){\circle{4}}
\put( 80,132){\scriptsize{\fbox{$16_1$}}}
\put(106,136){\line(1,1){14}}
\put(104,136){\line(-1,2){11.7}}
\put(121,152){\circle{4}}
\put(126,150){\scriptsize{$9_3$}}
\put(121,154){\line(0,1){14}}
\put(121,170){\circle{4}}
\put(126,168){\scriptsize{$4_4$}}
\put( 91,161){\circle{4}}
\put( 79,159){\scriptsize{$8_1$}}
\put(106,186){\line(1,-1){14}}
\put(104,186){\line(-1,-2){11.7}}
\put(105,188){\circle{4}}
\put( 84,186){\scriptsize{\fbox{$1_3$}}}
\put(105,190){\line(0,1){16}}
\put(105,208){\circle{4}}
\put(110,206){\scriptsize{$4_2$}}
\put(105,210){\line(0,1){16}}
\put(105,228){\circle{4}}
\put(110,226){\scriptsize{$2_3$}}
\put(105,230){\line(0,1){16}}
\put(105,248){\circle{4}}
\put(110,246){\scriptsize{$1_1$}}
\put(174,  0){$2a>b>a$}
\put(195, 20){\circle{4}}
\put(200, 18){\scriptsize{$1_4$}}
\put(195, 22){\line(0,1){10}}
\put(195, 34){\circle{4}}
\put(200, 32){\scriptsize{$2_4$}}
\put(195, 36){\line(0,1){10}}
\put(195, 48){\circle{4}}
\put(200, 46){\scriptsize{$4_5$}}
\put(195, 50){\line(0,1){10}}
\put(195, 62){\circle{4}}
\put(200, 60){\scriptsize{$2_2$}}
\put(195, 64){\line(0,1){10}}
\put(195, 76){\circle{4}}
\put(181, 74){\scriptsize{$9_4$}}
\put(197, 77){\line(3,2){15.7}}
\put(215, 88){\circle{4}}
\put(220, 86){\scriptsize{$8_4$}}
\put(215, 90){\line(0,1){10}}
\put(215,102){\circle{4}}
\put(220,100){\scriptsize{$1_2$}}
\put(215,104){\line(0,1){10}}
\put(215,116){\circle{4}}
\put(220,114){\scriptsize{$4_3$}}
\put(215,118){\line(0,1){10}}
\put(215,130){\circle{4}}
\put(220,128){\scriptsize{$9_2$}}
\put(195,142){\circle{4}}
\put(168,140){\scriptsize{\fbox{$16_1$}}}
\put(197,141){\line(3,-2){15.7}}
\put(180,109){\circle{4}}
\put(167,107){\scriptsize{$8_2$}}
\put(180,111){\line(1,2){14.5}}
\put(180,107){\line(1,-2){14.5}}
\put(197,143){\line(3,2){15.7}}
\put(215,154){\circle{4}}
\put(220,152){\scriptsize{$9_3$}}
\put(215,156){\line(0,1){10}}
\put(215,168){\circle{4}}
\put(220,166){\scriptsize{$4_4$}}
\put(215,170){\line(0,1){10}}
\put(215,182){\circle{4}}
\put(220,180){\scriptsize{$1_3$}}
\put(215,184){\line(0,1){10}}
\put(215,196){\circle{4}}
\put(220,194){\scriptsize{$8_3$}}
\put(195,208){\circle{4}}
\put(183,206){\scriptsize{$9_1$}}
\put(197,207){\line(3,-2){15.7}}
\put(180,175){\circle{4}}
\put(168,173){\scriptsize{$8_1$}}
\put(180,177){\line(1,2){14.5}}
\put(180,173){\line(1,-2){14.5}}
\put(195,210){\line(0,1){10}}
\put(195,222){\circle{4}}
\put(200,220){\scriptsize{$2_1$}}
\put(195,224){\line(0,1){10}}
\put(195,236){\circle{4}}
\put(200,234){\scriptsize{$4_2$}}
\put(195,238){\line(0,1){10}}
\put(195,250){\circle{4}}
\put(200,248){\scriptsize{$2_3$}}
\put(195,252){\line(0,1){10}}
\put(195,264){\circle{4}}
\put(200,262){\scriptsize{$1_1$}}
\put(284,  0){$b>2a$}
\put(295, 20){\circle{4}}
\put(300, 18){\scriptsize{$1_4$}}
\put(295, 22){\line(0,1){16}}
\put(295, 40){\circle{4}}
\put(300, 38){\scriptsize{$2_4$}}
\put(296, 42){\line(1,1){15.5}}
\put(294, 42){\line(-1,1){15.5}}
\put(277, 59){\circle{4}}
\put(264, 57){\scriptsize{$1_2$}}
\put(313, 59){\circle{4}}
\put(318, 57){\scriptsize{$4_5$}}
\put(296, 76){\line(1,-1){15.5}}
\put(294, 76){\line(-1,-1){15.5}}
\put(295, 78){\circle{4}}
\put(301, 75){\scriptsize{$8_4$}}
\put(296, 80){\line(3,2){15.5}}
\put(294, 80){\line(-1,1){17}}
\put(313, 92){\circle{4}}
\put(318, 90){\scriptsize{$9_4$}}
\put(311, 93){\line(-5,4){33.5}}
\put(313, 94){\line(0,1){14}}
\put(313,110){\circle{4}}
\put(318,108){\scriptsize{$2_2$}}
\put(313,112){\line(0,1){14}}
\put(313,128){\circle{4}}
\put(318,126){\scriptsize{$8_2$}}
\put(296,140){\line(3,-2){15.5}}
\put(295,142){\circle{4}}
\put(266,140){\scriptsize{\fbox{$16_1$}}}
\put(276, 99){\circle{4}}
\put(264, 97){\scriptsize{$4_3$}}
\put(276,101){\line(0,1){18}}
\put(276,121){\circle{4}}
\put(264,119){\scriptsize{$9_2$}}
\put(294,140){\line(-1,-1){17}}
\put(296,144){\line(3,2){15.5}}
\put(294,144){\line(-1,1){17}}
\put(313,156){\circle{4}}
\put(318,154){\scriptsize{$8_1$}}
\put(313,158){\line(0,1){14}}
\put(313,174){\circle{4}}
\put(318,172){\scriptsize{$2_1$}}
\put(313,176){\line(0,1){14}}
\put(313,192){\circle{4}}
\put(318,190){\scriptsize{$9_1$}}
\put(311,191){\line(-5,-4){33.5}}
\put(296,204){\line(3,-2){15.5}}
\put(295,206){\circle{4}}
\put(301,204){\scriptsize{$8_3$}}
\put(276,163){\circle{4}}
\put(264,161){\scriptsize{$9_3$}}
\put(276,165){\line(0,1){18}}
\put(276,185){\circle{4}}
\put(264,183){\scriptsize{$4_4$}}
\put(294,204){\line(-1,-1){17}}
\put(296,208){\line(1,1){15.5}}
\put(294,208){\line(-1,1){15.5}}
\put(277,225){\circle{4}}
\put(264,223){\scriptsize{$1_3$}}
\put(313,225){\circle{4}}
\put(318,223){\scriptsize{$4_2$}}
\put(296,242){\line(1,-1){15.5}}
\put(294,242){\line(-1,-1){15.5}}
\put(295,244){\circle{4}}
\put(300,242){\scriptsize{$2_3$}}
\put(295,246){\line(0,1){16}}
\put(295,264){\circle{4}}
\put(300,262){\scriptsize{$1_1$}}
\end{picture} \\
$\quad$ A box indicates a two-sided cell with several irreducible 
components, given as follows:\\
$\quad \mbox{\fbox{$4_2$}}=\{2_1,2_3,4_2\},\quad  
\mbox{\fbox{$4_5$}}=\{2_2,2_4,4_5\},\quad 
\mbox{\fbox{$1_3$}}=\{1_3,2_1,8_3,9_1\},\quad 
\mbox{\fbox{$1_2$}}=\{1_2,2_2,8_4,9_4\}$, \\
$\qquad\quad\mbox{\fbox{$12_1$}}=\{1_2,1_3,4_1,4_3,4_4,6_1,6_2,9_2,9_3,12_1,
16_1\},\quad \mbox{\fbox{$16_1$}}=\{4_1,6_1,6_2,12_1,16_1\}$.\\
$\quad$ Otherwise, the two-sided cell contains just one irreducible 
respresentation.
\end{tabular}
\end{center}}
\end{table}

\begin{proof} The weight function $L$ is specified by $a:=L(s_1)=L(s_2)>0$
and $b:=L(s_3)=L(s_4)>0$. We may assume without loss of generality that 
$b\geq a$. The preorder relations $\leq_{\cL}$, $\leq_{\cR}$, $\leq_{\cLR}$ 
and the corresponding equivalence relations on $W$ have been determined
in \cite{compf4}, based on an explicit computation of all the polynomials 
$p_{y,w}$ (where $y\leq w$ in $W$) and all polynomials $\mu_{y,w}^s$ (where 
$s \in S$ and $sy<y<w<sw$) using {\sf CHEVIE} \cite{chevie}. (The programs
are available upon request.) Once all this information is available, it is 
also a straightforward matter to check that condition ($*$) in 
Remark~\ref{proofp15} is satisfied, that is, {\bf P15} holds. Furthermore,
{\bf E3} and {\bf E4} are explicitly stated in \cite{compf4}. 

To check {\bf E1} and {\bf E2}, it is sufficient to use the information 
contained in Table~\ref{invF4} (which is taken from \cite[p.~318]{mylaus})
and Table~\ref{lrgraph} (which is taken from \cite[p.~362]{compf4}). In 
these tables, the irreducible representations of $W$ are denoted by $d_i$ 
where $d$ is the dimension and $i$ is an additional index; for example, 
$1_1$ is the trivial representation, $1_4$ is the sign representation and 
$4_2$ is the reflection representation. 

Thus, by Proposition~\ref{mainprop}, {\bf P1}--{\bf P14} also hold for $W$. 
(Note that, using similar computational methods, {\bf E1}--{\bf E4} are 
easily verified for all proper parabolic subgroups.)
\end{proof}

\begin{thm} \label{Mp115a} Lusztig's conjectures {\bf P1}--{\bf P15} hold in 
the following cases.
\begin{itemize}
\item[(a)] The {\em equal parameter case} where $\Gamma=\Z$ and $L(s)=a>0$
for all $s \in S$ (where $a$ is fixed).
\item[(b)] $(W,S)$ of type $B_n$, $F_4$ or $I_2(m)$ ($m$ even), with
weight function $L \colon W \rightarrow \Gamma$ given by:
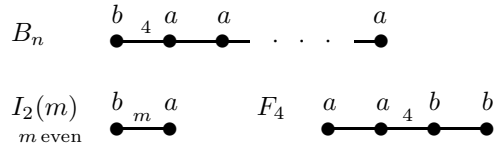
\begin{center}
\makeatletter
\vbox{\begin{picture}(200,60)
\put( 10,40){$B_n$}
\put( 50,40){\@dot{5}}
\put( 48,47){$b$}
\put( 50,40){\line(1,0){20}}
\put( 59,43){$\scriptstyle{4}$}
\put( 70,40){\@dot{5}}
\put( 68,47){$a$}
\put( 70,40){\line(1,0){30}}
\put( 90,40){\@dot{5}}
\put( 88,47){$a$}
\put(110,40){\@dot{1}}
\put(120,40){\@dot{1}}
\put(130,40){\@dot{1}}
\put(140,40){\line(1,0){10}}
\put(150,40){\@dot{5}}
\put(147,47){$a$}

\put( 10, 12){$I_2(m)$}
\put( 10, 02){$\;\scriptstyle{m\, {\rm even}}$}
\put( 50, 07){\@dot{5}}
\put( 48, 14){$b$}
\put( 56, 10){$\scriptstyle{m}$}
\put( 50, 07){\line(1,0){20}}
\put( 70, 07){\@dot{5}}
\put( 68, 14){$a$}

\put(103, 12){$F_4$}
\put(130, 07){\@dot{5}}
\put(128, 14){$a$}
\put(130, 07){\line(1,0){20}}
\put(150, 07){\@dot{5}}
\put(148, 14){$a$}
\put(150, 07){\line(1,0){20}}
\put(158, 10){$\scriptstyle{4}$}
\put(170, 07){\@dot{5}}
\put(168, 14){$b$}
\put(170, 07){\line(1,0){20}}
\put(190, 07){\@dot{5}}
\put(188, 14){$b$}
\end{picture}
\makeatother}
\end{center}
where $a,b \in \Gamma_{>0}$ are such that $b>ra$ for all $r\in\Z_{\geq 1}$.
\end{itemize}
\end{thm}

\begin{proof} (a) See Remark~\ref{equalgeo}. 
(b) For types $I_2(m)$ ($m$ even) and $F_4$, see Propositions~\ref{p115i2} 
and \ref{p115f4}. Now let $W$ be of type $B_n$ with parameters as above.  
The left, right and two-sided cells are explicitly determined by 
Bonnaf\'e and Iancu \cite{BI}, \cite{BI2}. A special feature of this case 
is that all left cells give rise to irreducible representations of $W$; see 
\cite[Prop.~7.9]{BI}; furthermore, two left cells give rise to isomorphic 
irreducible representations of $W$ if and only if they contained in the same 
two-sided cell; see \cite[\S 3]{BI2}. Based on these results, it is shown
in \cite[Theorem~1.3]{geia06} that {\bf P1}--{\bf P15} hold except possibly 
{\bf P9}, {\bf P10}, {\bf P15}. In \cite[Theorem~5.13]{myrel06}, the 
following implication is shown for all $x,y \in W$:
\begin{equation*}
x \sim_{\cLR} y \quad \mbox{and} \quad x \leq_{\cL} y  \quad
\Rightarrow \quad x \sim_{\cL} y.\tag{$\heartsuit$}
\end{equation*}
This then yields {\bf P9}, {\bf P10}; see \cite[Cor.~7.12]{myrel06}. 
Finally, {\bf P15$^\prime$} is shown in \cite[Prop.~7.6]{myrel06} under the 
additional assumption that $y \sim_{\cL} x' \sim_{\cR} w^{-1}$. However, if 
this additional assumption is not satisfied, then one easily sees, using 
{\bf P9} and {\bf P10}, that both sides of {\bf P15$^\prime$} are zero. Thus, 
{\bf P15$^\prime$} holds in general and then Lemma~\ref{Mp15prime} is used 
to deduce that {\bf P15} also holds. 
\end{proof}

\begin{cor} \label{Mp115b} Assume that $W$ is finite and let $L_0 \colon W 
\rightarrow \Gamma_0$ be the ``universal'' weight function of 
Remark~\ref{Mrem12}. Then {\bf P1}--{\bf P15} hold for at least one 
monomial order on $\Gamma_0$
where $L_0(s)>0$ for all $s\in S$.
\end{cor}

\begin{proof} By standard reduction arguments, we can assume that $(W,S)$
is irreducible. If $(W,S)$ if of type $B_n$, $F_4$ or $I_2(m)$ ($m$ even),
we choose a monomial order as in Theorem~\ref{Mp115a}(b). Otherwise, we
are automatically in the equal parameter case. Hence {\bf P1}--{\bf P15}
hold by Theorem~\ref{Mp115a}(a).
\end{proof}

Finally, we can show that Theorem~\ref{MisoH} holds without using
the hypothesis that {\bf P1}--{\bf P15} are satisfied!

\begin{cor} \label{Mp115c} Let $R \subseteq \C$ be a field. Then the
statements in Theorem~\ref{MisoH} hold for any weight function $L \colon 
W \rightarrow \Gamma$ where $\Gamma$ is an abelian group such that 
$A=R[\Gamma]$ is an integral domain. 
\end{cor}

Note that this implies Theorem~\ref{Mmain}, as stated in the introduction.

\begin{proof} Let $\Gamma_{0}$, $A_0$ and $\bH_0$ be as in
Remark~\ref{Mrem12}. To distinguish $A_0$ from $A$, let us write the
elements of $A_0$ as $R$-linear combinations of $\varepsilon_0^g$ where
$g \in \Gamma_0$. By Corollary~\ref{Mp115b}, we can choose a monomial
order $\leq$ on $\Gamma_0$ such that {\bf P1}--{\bf P15} hold. Let $\psi_0 
\colon \bH_0 \rightarrow A_0[W]$ be the corresponding homomorphism of 
Theorem~\ref{MisoH}.

Let $Q_0$ be the matrix of the $A_0$-linear map $\psi_0$ with respect 
to the standard $A_0$-bases of $\bH_0$ and $A_0[W]$. Let $\theta_0 
\colon A_0 \rightarrow R$ be the unique ring homomorphism such that 
$\theta_0(\varepsilon_0^g)=1$ for all $g \in \Gamma_0$. We denote by
$\theta_0(Q_0)$ the matrix obtained by applying $\theta_0$ to all 
entries of $Q_0$. By Theorem~\ref{MisoH}, $\theta_0(Q_0)$ is the identity
matrix.

Now, there is a group homomorphism $\alpha\colon \Gamma_0 \rightarrow
\Gamma$ such that $\alpha((n_s)_{s \in S})=\sum_{s \in S} n_sL(s)$. This
extends to a ring homomorphism $A_0\rightarrow A$ which we denote by
the same symbol. Extending scalars from $A_0$ to $A$ (via $\alpha$), we
obtain $\bH=A\otimes_{A_0} \bH_0$ and $A[W]=A \otimes_{A_0} A_0[W]$.
Furthermore, $\psi_0$ induces an algebra homomorphism $\bar{\psi}_0
\colon \bH\rightarrow A[W]$. Let $Q:=\alpha(Q_0)$ be the matrix obtained 
by applying $\alpha$ to all entries of $Q_0$. Then, clearly, $Q$ is the 
matrix of the $A$-linear map $\bar{\psi}_0$ with respect to the standard 
$A$-bases of $\bH$ and $A[W]$. 

Let $\theta_1 \colon A \rightarrow R$ be the unique ring homomorphism 
such that $\theta_1(\varepsilon^g)=1$ for all $g \in \Gamma$. As in the 
proof of Theorem~\ref{MisoH}, it remains to show that, if we apply $\theta_1$
to all entries of $Q$, then we obtain the identity matrix. But, we certainly
have $\theta_0=\theta_1 \circ \alpha$ and, hence, $\theta_1(Q)=
\theta_1(\alpha(Q_0))=\theta_0(Q_0)$. So it remains to recall that the
latter matrix is the identity matrix. 
\end{proof}

\noindent {\bf Acknowledgements}. I wish to thank Dean Alvis for providing me 
with the information on the partial order on two-sided cells in type $H_4$
needed in Example~\ref{h34}. (This information can be obtained from the data 
produced in \cite{Alvis87}.) 

 

\begin{thebibliography}{131} 
 
\bibitem{Alvis87}
{\sc D.~Alvis}, The left cells of the {C}oxeter group of type ${H}_4$, J.
Algebra {\bf 107} (1987), 160--168; the data produced in this article are 
electronically available at {\tt http://mypage.iusb.edu/$\sim$dalvis/h4data}.

\bibitem{AlLu82}
{\sc D.~Alvis and G.~Lusztig}, {The representations and generic degrees of the
Hecke algebra of type ${H}_4$}, J. Reine Angew. Math. \textbf{336} (1982),
201--212; correction, {\em ibid.} \textbf{449} (1994), 217--218.

\bibitem{BI2}
{\sc C.~Bonnaf\'e},  Two-sided cells in type $B$ in the asymptotic
case, J. Algebra {\bf 304} (2006), 216--236.

\bibitem{BI}
{\sc C.~Bonnaf\'e and L.~Iancu}, Left cells in type $B_n$ with unequal
parameters, Represent. Theory {\bf 7} (2003), 587--609.

\bibitem{bour}
{\sc N. Bourbaki}, Groupes et alg{\`e}bres de {L}ie, Chap. 4, 5
et 6, Hermann, Paris, 1968.


\bibitem{Fokko}
{\sc F.~DuCloux}, Positivity results for the Hecke algebras of
noncrystallographic finite Coxeter group, J. Algebra {\bf 303} (2006),
731--741.

\bibitem{my02}
{\sc M.~Geck}, Constructible characters, leading coefficients and left cells
for finite Coxeter groups with unequal parameters, Represent. Theory {\bf 6}
(2002), 1--30 (electronic).

\bibitem{compf4}
{\sc M.~Geck}, Computing Kazhdan--Lusztig cells for unequal parameters,
J. Algebra {\bf 281} (2004), 342--365; section "Computational Algebra".

\bibitem{myrel06}
{\sc M. Geck}, Relative Kazhdan--Lusztig cells, Represent. Theory {\bf 10} 
(2006), 481--524.  

\bibitem{mylaus}
{\sc M. Geck}, Modular representations of Hecke algebras, In: Group
representation theory (EPFL, 2005; eds. M. Geck, D.  Testerman and
J. Th\'{e}venaz), EPFL Press (2007), pp. 301-353.


\bibitem{myedin08}
{\sc M. Geck}, Leading coefficients and cellular bases of Hecke algebras,
Proc. Edinburgh Math. Soc., {\em to appear}.

\bibitem{chevie}
{\sc M. Geck, G.~Hi{\ss}, F.~L\"ubeck, G.~Malle and G.~Pfeiffer},
{\sf CHEVIE}-A system for computing and processing generic character tables
for finite groups of Lie type, Weyl groups and Hecke algebras. Appl.
Algebra Engrg. Comm. Comput. {\bf 7} (1996), 175--210.

\bibitem{geia06}
{\sc M. Geck and L.~Iancu}, Lusztig's $a$-function in type $B_n$ in
the asymptotic case. Special issue celebrating the $60$th birthday of
George Lusztig, Nagoya J. Math. {\bf 182} (2006), 199--240.

\bibitem{gepf} 
{\sc M. Geck and G. Pfeiffer}, Characters of finite Coxeter groups and 
Iwahori--Hecke algebras, London Math. Soc. Monographs, New Series {\bf 21}, 
Oxford University Press, New York 2000. xvi+446 pp. 
 
\bibitem{KaLu}
{\sc D.~A.~Kazhdan and G.~Lusztig}, {Representations of {C}oxeter groups and
{H}ecke algebras}, Invent. Math. {\bf 53} (1979), 165--184.

\bibitem{KaLu2}
{\sc D.~A. Kazhdan and G.~Lusztig}, {S}chubert varieties and {P}oincar\'e
duality, Proc. Sympos. Pure Math. {\bf 34} (1980), 185--203.

\bibitem{Lu0}
{\sc G.~Lusztig}, {On a theorem of {B}enson and {C}urtis},
J. Algebra {\bf 71} (1981), 490--498.


\bibitem{Lusztig83}
{\sc G.~Lusztig}, Left cells in {W}eyl groups, {\em Lie Group
Representations, I} (eds R.~L. R.~Herb and J.~Rosenberg), Lecture Notes
in Mathematics 1024 (Springer, Berlin, 1983), pp.~99--111.

\bibitem{LuBook}
{\sc G.~Lusztig}, Characters of reductive groups over a finite field,
Annals Math.\ Studies, vol. 107, Princeton University Press, 1984.

\bibitem{Lu1}
{\sc G.~Lusztig}, Cells in affine Weyl groups, Advanced Studies in Pure
Math. {\bf 6}, Algebraic groups and related topics, Kinokuniya and
North--Holland, 1985, 255--287.

\bibitem{Lu2}
{\sc G. Lusztig}, Cells in affine Weyl groups II, J. Algebra {\bf 109} 
(1987), 536--548.


\bibitem{Lusztig03}
{\sc G.~Lusztig}, Hecke algebras with unequal parameters, CRM Monographs
Ser.~{\bf 18}, Amer. Math. Soc., Providence, RI, 2003.

\bibitem{Spr}
{\sc T. A.~Springer}, Quelques applications de la cohomologie d'intersection,
S\'{e}minaire Bourbaki (1981/82), exp. 589, Ast\'erisque {\bf 92--93} (1982).
\end{thebibliography}
\end{document}